\documentclass[11pt,letterpaper]{article}
\usepackage{tikz}
\usetikzlibrary{arrows.meta,positioning}
\usepackage{tikz-cd}
\usepackage{amssymb}
\usepackage[T1]{fontenc}
\usepackage[utf8]{inputenc}
\usepackage{babel}
\usetikzlibrary{babel}
\usepackage{booktabs}
\usepackage{xcolor}
\usepackage{geometry}
\usepackage{mathrsfs}
\usepackage{amsfonts,amsmath,lmodern,fancyhdr,lastpage,graphicx,nccfoots,caption}
\usepackage{afterpage}

\usepackage{blkarray, bigstrut}
\usepackage{tikz}
\usepackage[all]{xy}
\usepackage{capt-of}
\usepackage{graphicx}
\usepackage{enumitem}
\usepackage{amssymb}
\usepackage{amsmath}
\usepackage{amsfonts}
\usepackage{amsthm}
\usepackage{amscd}
\usepackage{epstopdf}
\usepackage{mathtools}
\usepackage{color}

\newcommand{\FF}{\mathbb{F}}
\newcommand{\ZZ}{\mathbb{Z}}
\newcommand{\NN}{\mathbb{N}}

\newcommand{\CC}{\mathbb{C}}
\newcommand{\PP}{\mathbb{P}}

\DeclareMathOperator{\ord}{ord}

\newenvironment{definition}
{ \begin{flushleft} \textbf{Definition:}}
	{ \end{flushleft} }

\newenvironment{remark}
{ \begin{flushleft} \textbf{Remark:}}
	{ \end{flushleft} }

\newtheorem{lemma}{Lemma}

\newtheorem{theorem}{Theorem}

\newtheorem{corollary}{Corollary}

\newtheorem{proposition}{Proposition}

\newtheorem{conjecture}{Conjecture}

\usepackage{amsmath}

\newtheorem{lettertheorem}{Theorem} 
\newtheorem{letterproposition}{Proposition} 

\vfuzz \hfuzz
\newcounter{a}
\ifodd\thea\else\stepcounter{a}\fi

\usepackage{hyperref}
\usepackage{cleveref}

\begin{document}
	\thispagestyle{plain}
	\begin{center}
		\Large
		\textsc{Surfaces of general type with extremal cotangent dimension }
	\end{center}

	\begin{center} \itshape Dedicated to Fedya on the occasion of his eightieth birthday \end{center}
	\begin{center}
 \textit{Damian Brotbek}
 \footnote{Damian Brotbek \\
Université de Lorraine,
CNRS UMR 7502, IECL, F-54000 Nancy, France\\ 
email: damian.brotbek@univ-lorraine.fr}
\smallskip, 
\textit{Bruno De Oliveira}\footnote {Bruno De Oliveira\\University of Miami, Coral Gables, 1365 Memorial Dr
Miami, FL 33134\\email:bdeolive@miami.edu}
\smallskip, 
  \textit{Erwan Rousseau} \footnote{Erwan Rousseau \\
Univ Brest, CNRS UMR 6205, Laboratoire de Mathematiques de Bretagne Atlantique, F-29200 Brest, France \\ email: erwan.rousseau@univ-brest.fr}

	\end{center}
	
	\noindent
	\textbf{Abstract} 
    We study the geography of surfaces of general type  with extremal cotangent dimension, in other words surfaces which have either no global holomorphic symmetric differentials at all or the maximal asymptotic growth of their number, i.e. big cotangent bundle. We are mainly interested in surfaces with low slope $K^2/\chi$, which, as far as maximal cotangent dimension is concerned, were out of reach of previous methods. We prove vanishing theorems for symmetric logarithmic differentials on minimal rational surfaces and for differentials on their double covers and extend a bigness criterion of Sakai to fibrations of general type in the sense of Campana. As a consequence, we prove, on the one hand that generic Horikawa surfaces have no nontrivial symmetric differentials and on the other hand that there exist Horikawa surfaces with big cotangent bundle.
	\medskip

	\noindent\textbf{keywords:} 
	big cotangent bundle; cotangent dimension; surfaces of general type; hyperbolicity; symmetric differentials; double covers; genus 2 fibrations; Horikawa surfaces

\section{Introduction}\label{intro}

The $(K^2,\chi)$-geography of minimal surfaces of general type with maximal cotangent dimension, or in other words with big cotangent bundle, has been a topic of research since the seminal work of Bogomolov \cite{bogomolov_1979} where it was shown that a surface $X$ with maximal cotangent dimension has a proper subvariety $Z\subsetneq X$ containing all its rational and elliptic curves. 

\

This result is closely related to the Green-Griffiths-Lang conjecture.
\begin{conjecture}
Let $X$ be a projective variety of general type. Then there exists a proper subvariety $Z\subsetneq X$ such that any non-constant holomorphic map $f: \CC \to X$ has its image $f(\CC) \subset Z$.
\end{conjecture}

McQuillan \cite{mcquillan} extended Bogomolov's result to the analytic setting and showed that surfaces with maximal cotangent dimension satisfy the Green-Griffiths-Lang conjecture, i.e. the proper subset $Z\subsetneq  X$ contains also all entire curves of $X$ (i.e. non-constant holomorphic maps $f: \CC \to X$).

As a consequence of an easy Riemann-Roch computation, Bogomolov and McQuillan's results apply to surfaces with $c_1^2 > c_2$.

\

The works \cite {bogomolov_tensors}, \cite{sakai},\cite{bogomolov_nodes},\cite{roulleau2014},\cite{bruin2022explicit} and \cite{ADOWI} describe several conditions involving topological data, e.g. Chern numbers, singularity data, e.g. the singularities of the canonical model and existence of special fibrations that guarantee  maximal cotangent dimension for surfaces.  If we view these conditions from the $(K^2,\chi)$-geographic perspective, it stands out that they all necessarily fail to hold in the  region $K^2/\chi<2$. In section 1 we give a presentation of  the results, just mentioned, that are relevant to frame our work.

\ 

This work presents a framework and tools to determine whether a surface has maximal or minimal cotangent dimension (the latter corresponds to the absence of nontrivial symmetric differentials) that are, in particular, suitable for surfaces with low $K^2/\chi$. We apply our results to the class of (Noether) Horikawa surfaces ($K^2=2\chi-6$), which are the surfaces with the lowest $K^2/\chi$ slope for any fixed $\chi \geqslant 4$. 

\ 

In section 2, we present a filtration, concerning the order of the poles, for the symmetric power $S^m\Omega^1_X(\log D)$ where $X$ is a smooth projective variety and $D$ a simple normal crossing divisor and characterize the graded pieces of the filtration via higher degree residue exact sequences. A consequence of this approach is showing that $H^0(X,S^m\Omega^1_X(\log D))=0$ for $m>0$ if $X$ is a minimal rational surface and $D$ is an irreducible and smooth divisor, see \cref {cor:VanishingP2} and \cref{prop:VanishingFd1Comp}.

\ 

Minimal surfaces of general type with $K^2/\chi<2$, with the exception of surfaces in finitely many families having $\chi\leqslant 9$, are double covers of smooth (not necessarily minimal) rational surfaces \cite{beauville}.  In section 3, we obtain vanishing theorems for the symmetric differentials on double covers of minimal rational surfaces ($\PP^2$ and $\FF_d$, the results also include the case $d=1$ which is not minimal).

\

\begin{lettertheorem}\label{thm:A} (\cref{A}) Let \(X\) be a smooth double cover of \(\PP^2\) or \(\FF_d\), then $X$ will not have maximal cotangent dimension (i.e. $\Omega^1_X$ cannot be big). Moreover, unless the base of the double cover is $\FF_d$ and the branch locus is a union of fibers, then $X$ has minimal cotangent dimension, i.e. 
    \[H^0(X,S^m\Omega_X^1)=0,\ \ \forall m\geqslant 1.\]
\end{lettertheorem}

\

In \cref{CMS} it is shown that there are smooth double covers of  non-minimal rational surfaces with maximal cotangent dimension (this result involves the CMS-bigness criterion  \cite{ADOWI} described in section 1). 

\

Horikawa surfaces can be viewed as minimal resolutions of double covers of $\PP^2$ or $\mathbb{F}_d$ branched along a possibly reducible simple
 curve (a curve with only simple singularities, i.e., of type $a$, $d$, or $e$). We  establish a vanishing theorem for symmetric differentials on a generic Horikawa surface.

\begin{lettertheorem}\label{thm:B} (\cref{generic}) The generic Horikawa surface $X$ has no nontrivial symmetric differentials, i.e., for all \(m\geqslant 1\),
$$H^0(X,S^m\Omega^1_X)=0.$$
\end{lettertheorem}

\ 

The meaning of generic in the above theorem is stronger than the Horikawa surface corresponding to a generic point in moduli space ${\cal H}_{p_g}$ of all Horikawa surfaces with a fixed $p_g$. It means that the Horikawa surface corresponds to a generic point of any of the strata of ${\cal H}_{p_g}^{(d)}\subset {\cal H}_{p_g}$ parameterizing all Horikawa surfaces that are  minimal resolutions of double covers of $\mathbb{F}_d$ if $d\in \ZZ_\ge 0$ and of $\PP^2$ if $d=\infty$ branched along a simple curve (\cite{Horikawa1976I}). To accomplish the above vanishing one needs  to consider, in addition to smooth double covers of $\PP^2$ and $\FF_d$ covered by theorem A, minimal resolutions of singular double covers of $\FF_d$ whose branch locus is singular with two smooth irreducible components intersecting transversely (via the canonical resolution process of Horikawa these are smooth double covers of non-minimal rational surfaces).

\ 

\begin{lettertheorem}\label{thm:C} (\cref{cor:VanishingCover2}) Let \(X\) be the minimal resolution of the  double cover of \(\FF_d\) branched along a simple normal crossing divisor of the form \(\Delta_0+D\), where \(\Delta_0\) is a section with self intersection $-d$ (unique if $d>0$) and \(D\) is an irreducible smooth divisor. Then 
    \[H^0(X,S^m\Omega_X^1)=0,\ \ \forall m\geqslant 1.\]
\end{lettertheorem}

\  

\medskip

In section 4, we investigate the existence of minimal surfaces of general type having maximal cotangent dimension in the $(K^2,\chi)$-region, $K^2/\chi<2$, where the known results fail to show the existence of nontrivial symmetric differentials. Our approach moves from viewing surfaces as covers to viewing them as fibrations. The result that allows us to reach the region of $K^2/\chi<2$ where the base of the fibration is necessarily $\PP^1$ is the following orbifold generalization of Sakai's criterion stating that a genus-2 fibration over a curve of genus at least $2$ has big cotangent bundle.

\begin{letterproposition}\label{prop:A}(\cref{orbibig}) Let $f: X \to C$ be a fibration with non-classical orbifold general type base on a projective surface $X$ with general fiber of genus $\geqslant 2$. Then $X$ has big cotangent bundle $\Omega^1_X$.
\end{letterproposition}

\

All Horikawa surfaces, with the exception of the members of the two families in ${\cal H}_{3}$ and ${\cal H}_{6}$ consisting of surfaces having $\PP^2$ as their canonical image, have genus-2 fibrations. Stoppino in \cite{Stoppino} produced examples of surfaces with genus-2 fibrations having  non-classical orbifold general type base, but these examples lie in the region $K^2\geqslant 2\chi$.     The non-classical orbifold structure of the base comes from the presence of  enough Campana multiple fibers (fibers whose non-classical multiplicity, i.e. the minimum of the multiplicity of its irreducible components, is greater than 1). There are four types of Campana multiple fibers of genus 2 appearing in the classification of genus 2 fibers by Ogg and Namikawa-Ueno,\cite{ogg} or \cite{nu}. Horikawa \cite{genus2} studied the impact of singular fibers of genus 2 in a fibration of genus 2 on the relationship between $K^2$ and $\chi$, $K^2=2\chi-6+\sum_{F_i} H(F_i)$ where $F_i$ are the singular fibers of the fibration and $H(F_i)\ge 0$ is the Horikawa invariant associated with a fiber of genus 2. 

\ 

We show that only one of the four types of Campana multiple fibers can occur in a Horikawa surface (i.e. it has vanishing Horikawa invariant) and called it the Campana multiple fiber of type $(0)$. We construct Horikawa surfaces with Campana multiple fibers of type $(0)$ as minimal resolutions of double covers of Hirzebruch surfaces whose branch locus is a simple curve whose set of singularities have the special configuration associated with the fibers of type $(0)$.  We then obtain  the existence of special Horikawa surfaces with maximal cotangent dimension.

\

\begin{lettertheorem}\label{thm:D} (\cref{max}) There exist Horikawa surfaces $X$ with big cotangent bundle $\Omega^1_X$.
\end{lettertheorem}

Our approach to the bigness of the cotangent bundle of Horikawa surfaces via the orbifold generalization of Sakai's criterion also requires the presence of singularities in their canonical model  (at least ten $E_7$ singularities). The difference and strength of our approach compared to previous criteria, e.g. the CMS-criterion, is that our approach takes into account the global geometry of the set of singularities and not just the set of singularities. 

\

In our examples of Horikawa surfaces $X$ with big cotangent bundle the degeneracy locus $Z\subset X$, i.e. the Zariski closure of the union of all the entire curves of $X$,  described in the Green-Griffiths-Lang conjecture,  is understood. It follows from \cref{locus} for relatively minimal genus 2 fibrations that $Z$ consists of the union of all the singular fibers of the genus 2 fibration on $X$ having a base of non-classical orbifold general type.

\ 

The authors learned many of the ideas underlying this work from Fedya. The second author had the privilege and great pleasure of having Fedya as his advisor and collaborator. We wish him all the best and express our heartfelt gratitude to him.

\

\section{Background and the geography of surfaces with big cotangent bundle}\label{bg}

\ 

\subsection{Cotangent dimension} \label{cd}

\

Let $X$ be a compact complex manifold and $E$ a vector bundle on $X$ of rank $r$.
Let $\pi\colon \mathbb{P}(E)\to X$ be the associated $\mathbb{P}^{r-1}$-bundle,
with $\mathcal{O}_{\mathbb{P}(E)}(1)$ its tautological line bundle. The vector
bundle $E$ is said to be ample, big, nef, \emph{etc.} if
$\mathcal{O}_{\mathbb{P}(E)}(1)$ is respectively ample, big, nef, \emph{etc.}

\

\begin{definition} The symmetric dimension of a vector bundle $E$  is:
\vspace{-.2in}

\begin{align*} \lambda(X,E):=k(\PP(E),\mathcal{O}_{\PP(E)}(1))-(r-1)
    \end{align*}
    
\end{definition}

\noindent where $k(Y,L)$ denotes the Iitaka dimension of the line bundle $L$ on the projective variety $Y$.              

\begin{remark} $\lambda(X,E)=-\infty,-\text{rank} E+1,-\text{rank} E+2,...,\dim X-1,\dim X$. The above symmetric dimension of $E$ differs from the $E-$dimension of Sakai \cite {sakai} just in the case $\lambda(X,E)=-\infty$ which has $E-$dimension= -rank$E$ (the notion of symmetric dimension as above  appeared for example in \cite{Manivel1995}).
\end{remark}

\begin{definition} The cotangent dimension of a smooth projective variety $X$ is $\lambda_\Omega(X):=\lambda(X,\Omega^1_X)$.
\end{definition}

We are interested in the extremal cases: (minimal) $\lambda_\Omega(X)=-\infty$ where $X$ has no nontrivial symmetric  differentials; (maximal) $\lambda_\Omega(X)=\dim X$, where $X$ has maximal abundance of symmetric differentials in the sense that $\Omega^1_X$ is big.

\

 \

The following are properties for the cotangent dimension we will use throughout the paper, a possible reference is \cite {sakai}.
 
 \

\noindent (Birational invariance) If $X$ and $Y$ are two birational smooth projective varieties, then $\lambda_\Omega(X)=\lambda_\Omega(Y)$. In fact, the symmetric plurigenera $P^S_m(X):=h^0(X,S^{m}\Omega^1_X)$ are birational invariants (among smooth varieties).

\begin{remark} If one considers also  normal singular models $X$ of a birational class $\mathcal{X}$, then  symmetric plurigenera is no longer a birational invariant, where the plurigenera of $X$ is defined as $P^S_m(X):=h^0(X,S^{[m]}\Omega^1_X)$ with $S^{[m]}$ the reflexive symmetric power. For example, if a surface $X$ has   log-terminal or even canonical singularities (mild singularities), then $P^S_m(X)=h^0(\tilde X\setminus E, S^m\Omega^1_{\tilde X})=h^0(\tilde X,S^m\Omega^1_{\tilde{X}}(log E))$ with $\tilde X$ the minimal resolution of $X$ and $E$ the exceptional locus, \cite{miyaoka1984maximal}, which can be distinct from $P_m^S(\tilde X)$.
    \end{remark}

\noindent (\'Etale invariance) If $X$ is an \'etale cover of a smooth projective variety $Y$, then $\lambda_\Omega(X)=\lambda_\Omega(Y)$.

\

\noindent (Maximal cotangent dimension under equidimensional dominant maps) Let $f:X\to Y$ be a surjective equidimensional map between smooth projective varieties. If $\lambda_\Omega(Y)=\dim Y$, then $\lambda_\Omega(X)=\dim X$.

\ 

\noindent (Easy addition formula) Let $f:X \to Y$ be a surjective connected morphism between smooth projective varieties with $X_g$ the general fiber, then $\lambda_\Omega(X)\leqslant \lambda(X_g,\Omega_X^1|_{X_g})+\dim Y$.

\ 

\subsection{Surfaces with maximal cotangent dimension}

 \

 \

 The canonical volume $ K^2$ of a minimal surface $X$ of general type satisfies the inequalities
 $2\chi-6\leqslant K^2\leqslant 9\chi$, where the lower bound is the Noether's inequality and the upper bound is the Bogomolov-Miyaoka-Yau's inequality. While surfaces with large slope, $K^2/\chi>6$ ($c_1^2>c_2)$ are known to have maximal cotangent dimension \cite{bogomolov_tensors}, surfaces with low slope $K^2/\chi$, say less than 2, are expected to have minimal cotangent dimension.

 \

If $X$ is a smooth projective surface, $X$ having maximal cotangent dimension can be also described by:

 $$h^0_\Omega(X):=\lim_{m \rightarrow \infty} \dfrac{h^0(X, S^m \Omega_X^1)}{m^{3}} \neq 0$$
	
	\noindent i.e. the symmetric pluri-genera $P^S_m(X):=h^0(X, S^m \Omega_X^1)$ have the maximal growth order possible with respect to $m$. The above limit exists  (no need to use limsup). This is due to  $h^0(X, S^m\Omega_X^1)=h^0(\mathbb P(\Omega^1_X),\mathcal{O}(m))$ and the volume of  a line bundle on a  projective variety being given by a regular limit  (\cite{Lazarsfeld2004PositivityIA} 11.4.7).

\

The Hirzebruch--Riemann--Roch theorem for symmetric powers of $\Omega_X^1$
and Bogomolov's vanishing for surfaces of general type
(\cite{bogomolov_stability}) 
\[
H^2\bigl(X,S^m\Omega_X^1\bigr)=0
\qquad \text{for } m \geqslant 3,
\]
give the following asymptotic statement:
\[
h_\Omega^{0}(X)
=
\frac{c_1^2(X)-c_2(X)}{3!}
+
h_\Omega^{1}(X),
\]
with
\[
h_\Omega^{1}(X)
=
\lim_{m\to\infty}
\frac{h^1\!\bigl(X,S^m\Omega_X^1\bigr)}{m^3}.
\] (note that this limit exists for surfaces of general type due to the asymptotic Hirzebruch--Riemann--Roch plus Bogomolov's vanishing).  Hence, $X$ having maximal cotangent dimension is equivalent to $\frac{c_1^2(X)-c_2(X)}{3!}
+
h_\Omega^{1}(X)>0$.

\

From the above easily follows:

\

\begin{proposition} \cite{bogomolov_stability}
A surface of general type has maximal cotangent dimension if
\[
c_1^2 > c_2
\qquad
(\text{equivalently } K^2 > 6\chi).
\]
\end{proposition}

\

If a minimal surface $X$ lies in the region $K^2/\chi \leqslant 6$, then having maximal
cotangent dimension is more delicate, since one needs  to take into account
$h_\Omega^1(X)$, which is dependent on the complex structure. In particular, if $K^2/\chi \leqslant 6$, maximal cotangent dimension is not
necessarily preserved under deformation \cite{bogomolov_nodes}.

\medskip

To get a hold on $h_\Omega^1(X)$, an approach initiated in
\cite{bogomolov_nodes} consists of giving a lower bound for
$h_\Omega^1(X)$ coming from the singularities of the canonical model
$X_{\mathrm{can}}$ of $X$, see also \cite{roulleau2014},\cite{bruin2022explicit} and \cite{ADOWI}.

\medskip

The singularity invariants involved in the lower bound for $h_\Omega^1(X)$
are the following: let $x$ be a canonical surface singularity
(or equivalently $x$ is an ADE singularity), and let $U_x$ be an affine surface
with a single singularity of type $x$. Then define
\[
h_\Omega^1(x)
=
\liminf_{m \to \infty}
\frac{h^0\bigl(U_x, R^1\sigma_* S^m\Omega_X^1\bigr)}{m^3}.
\]

\noindent where $\sigma:\tilde {U_x} \to U_x$ is the minimal resolution. If $x$ is of type $A$, then the $\liminf$ is actually a regular  limit \cite{ADOWII} (the same is expected for type D and E).

\

The lower bound for $h_\Omega^1(X)$ for a minimal surface of general type $X$
that we obtain is
\begin{align}
L_\Omega^1(X):=\sum_{x\in \mathrm{Sing}(X_{\mathrm{can}})} h_\Omega^1(x),
\end{align}
from which follows:

\medskip

\noindent
\textbf{(CMS--criterion)} \cite{ADOWI} Let $X$ be a minimal surface of general type. Then $X$ has maximal cotangent
dimension if
\begin{align}\label{cms}
L_\Omega^1(X) + \frac{c_1^2(X)-c_2(X)}{3!} > 0.
\end{align}

\medskip

The CMS-criterion gave the best known results for regular surfaces with maximal cotangent dimension to date. The singularity data required by the CMS-criterion doesn't involve (global) information on the geometry of the singularities arrangement, it just involves the type and number of singularities. This, in particular, makes it easy to apply and to understand its reach, as illustrated ahead. On the other hand, in \cref{max} we will give a construction where the maximal cotangent dimension holds and it is  beyond the reach of the CMS-criterion. This result is heavily dependent of the geometry of the singularities arrangement, which the CMS-criterion does not take into account.  

\

In \cite{ADOWII} it was shown that the invariants $h_\Omega^1(x)$, where $x$ is
an $A_n$ singularity, denoted by $h_\Omega^1(A_n)$, are described via a closed
formula in $n$, this formula gives, in particular, that

\begin{align} \label{h-bound}
    \frac{n}{6}-\frac{1}{45} < h_\Omega^1(A_n) < \frac{n}{6}.
\end{align}

\medskip

The lower bound for $h_\Omega^1(A_n)$ is useful to provide examples of surfaces
with maximal cotangent dimension \cite{ADOWI}.

\

The upper bound in \cref{h-bound} restricts the $(K^2,\chi)$-region
where the CMS-criterion can be applied, for surfaces of type~A, i.e.\ whose
canonical model only has canonical singularities of type~A. The restriction originates from bounds on the possible singularity
collections on $X_{\mathrm{can}}$ for a minimal surface of general type
$X$ with given $K^2$ and $\chi$. There are two general bounds:

\medskip

\noindent
(Miyaoka bound)
\[
\sum_{x \in \mathrm{Sing}(X_{\mathrm{can}})} c_2(x)
\;\le\;
c_2(X) - \frac{1}{3}c_1^2(X).
\]

\medskip

\noindent
(Hodge bound)
\begin{align}\label{hodge}
   \rho_{-2}(X)\le\frac{1}{6}\bigl(5c_2(X) - c_1^2(X)\bigr) + b_1(X) - 1.
\end{align}

\medskip

\noindent where $c_2(x)$ is the local second Chern number of the singularity $x$, e.g. $c_2(A_n)=n+1-\frac{1}{n+1}$, and $\rho_{-2}(X) = \#\{(-2)\text{-curves on } X\}$.

\medskip

These bounds  give an upper bound for $L_\Omega^1(X)$ from
the invariants $K^2$ and $\chi$. The upper bound coming from the Hodge bound is the strongest of the two for surfaces of type~A in the region $\frac{K^2}{\chi} \leqslant 6$ and the following holds

\begin{proposition} \cite{ADOWI} \label{fails}
Let $X$ be a minimal surface of type~A in the region
$\frac{K^2}{\chi} \leqslant 2$,
then
\[
L_\Omega^1(X) + \frac{c_1^2 - c_2}{3!} < 0
\]
(i.e. $X$ is beyond the reach of the  CMS-criterion).
    
\end{proposition}

\medskip

For (minimal) regular surfaces, the example that satisfies the CMS-criterion with the lowest $K^2/\chi$ ratio known ($K^2/\chi=25/11$) is presented in \cref{CMS}. This example is relevant to the theme of this work also because the surface is a double cover of a non minimal rational surface, the non minimality condition  must be present since we show in section 3 that a smooth double cover of a minimal rational surface cannot have maximal cotangent dimension.

\

For irregular surfaces, the best and final result follows from a distinct approach using:

\medskip

\begin{proposition} (\cite{sakai})\label{sakai}
Let $X$ be a surface of general type fibered over a curve $C$ with genus
\[
g(C) \geqslant 2.
\]
Then $X$ has maximal cotangent dimension.
\end{proposition}

\medskip

This proposition cannot be applied in the region $K^2/\chi \leqslant 2$, since
$q(X) = 0 \quad \text{if } \frac{K^2}{\chi} < 2$ (\cite{Bombieri}) and
$q(X) \leqslant 1 \quad \text{if } \frac{K^2}{\chi} = 2$  (\text{see \cite{Horikawa1979}}), but implies:

\begin{proposition}

\[
\alpha^+_{\mathrm{big}}
:=
\inf
\left\{
\frac{K^2}{\chi}(X)
\;\middle|\;
\lambda_{\Omega}(X)=2,\; q(X)\neq 0,\;
X  \text{ minimal surface}
\right\}
=2.
\]
\end{proposition}

\medskip

\begin{proof}
This result follows from  \cref{sakai}
and the following standard construction of genus $2$ fibrations.

\medskip

Let $C$ be a smooth curve of genus $g$ and
$E$ a very ample rank $2$ vector bundle over $C$.

Consider the $\mathbb{P}^1$-bundle
\[
p \colon \mathbb{P}(E) \longrightarrow C.
\]

Let
\[
H \in \left|\mathcal{O}_{\mathbb{P}(E)}(1)\right|
\]
and
\[
K_{\mathbb{P}(E)}
=
-2H + p^*(K_C + \det E)
\]
be respectively a tautological divisor and a canonical divisor
of $\mathbb{P}(E)$.

The following holds:
\[
H^2 = \deg E,
\]
\[
K_{\mathbb{P}(E)} \cdot H
=
-\deg E + 2(g-1),
\]
\[
K_{\mathbb{P}(E)}^2 = 8(g-1).
\]

Consider the double cover $\pi \colon X_E \to \mathbb{P}(E)$ branched along $B \in |6H|$
smooth (possible since $H$ is very ample). The surface $X_E$ by construction is a genus $2$
fibration over $C$ and hence has $\lambda_\Omega(X_E)=2$ by \cref{sakai}.

\medskip

We have 
\[
K_{X_E} = \pi^*(K_{\mathbb{P}(E)} + 3H)
\]
and the following formulas hold (\cite[V.22]{BPV}):

\[
\chi(X_E)
=
2\chi(\mathbb{P}(E))
+
\frac{1}{2}K_{\mathbb{P}(E)}\cdot 3H
+
\frac{1}{2}(3H)^2
=
3\deg E + g -1,
\]

\[
K_{X_E}^2
=
2K_{\mathbb{P}(E)}^2
+
4K_{\mathbb{P}(E)}\cdot 3H
+
2(3H)^2
=
6\deg E + 8(g-1).
\]

Hence
\[
\frac{K_{X_E}^2}{\chi(X_E)}
=
\frac{6\deg E + 8(g-1)}
{3\deg E + g -1}
>2,
\]
and since $\deg E$ can be arbitrarily large,
the result follows.

\end{proof}

\

\

\

\section{Symmetric log-differential and higher degree residue exact sequences}

\

\

Let \(X\) be a smooth projective variety, \(D=D_1+\cdots+D_c\subset X\) be a simple normal crossing divisor with \(c\in \NN\) components, and \(\Omega_X^1(\log D)\) denote the logarithmic cotangent bundle. Recall that this bundle is the locally free sheaf defined as the subsheaf of the sheaf of meromorphic \(1\)-forms with at most logarithmic poles along \(D\). 
That is to say, in any open subset \(U\) with \(\sigma_i\in \mathcal{O}(U)\) such that \(D_i\cap U\) is given by \((\sigma_i=0)\) for all \(1\leqslant i\leqslant c\), one has
\[\Omega_X^1(\log(D))(U)=\left\{\frac{d\sigma_1}{\sigma_1}f_1 +\cdots+ \frac{d\sigma_c}{\sigma_c}f_c +\omega\ ;\ \  f_1,\dots, f_c\in \mathcal{O}(U), \ \ \omega\in \Omega_X^1(U)\right\}.\]

In order to compute the cohomology of these sheaves in an inductive way on the number of components, we introduce  a natural filtration on the pole order along any chosen component (here the component \(D_c\) for simplicity). 
Let us write \(D'=D_1+\cdots+D_{c-1}\), so that \(D=D'+D_c\). 
Let \(m\in \mathbb{N}\) and consider the \(m\)-th symmetric power \(S^m\Omega_X^1(\log D)\). There exists an increasing filtration \(S^m\Omega_X^1(\log D)_{(\bullet),D_c}\) on  \(S^m\Omega_X^1(\log D)\)
defined as follows: For any \(0\leqslant k\leqslant m\), the sheaf \(S^m\Omega_X^1(\log D)_{(k),D_c}\) is the subsheaf of \(S^m\Omega_X^1(\log D)\) of forms with pole order at most \(k\) along \(D_c\).
That is to say, for any open subset \(U\) over which \(D_1,\dots,D_c\) are defined by \((\sigma_1=0),\dots,(\sigma_c=0)\) for some \(\sigma_1,\dots, \sigma_c\in \mathcal{O}({U})\) as above, 
\[S^m\Omega_X^1(\log D)_{(k),D_c}(U):=\left\{ \omega\in S^m\Omega_X^1(\log D)(U); \ \ \sigma_c^k\omega\in S^m\Omega_X^1(\log D')(U)\right\}.\]

Over such an open subset, one can write any element \(\omega\in S^m\Omega_X^1(\log D)_{(k),D_c}(U)\) in the form
\[\omega=\left(\frac{d\sigma_c}{\sigma_c}\right)^k\alpha+\beta\]
where \(\alpha\in S^{m-k}\Omega_X^1(\log D')(U)\) and \(\beta\in S^m\Omega_X^1(\log D)_{(k-1),D_c}(U)\), i.e. \(\sigma_c^{k-1}\beta\in S^m\Omega_X^1(\log D')(U)\).


Observe however that the decomposition \(\omega=\left(\frac{d\sigma}{\sigma}\right)^k\alpha+\beta\) need not be unique. It is not hard to see that \(S^m\Omega_X^1(\log D)_{(k),D_c}\) are actually locally free sheaves. Indeed, one can suppose that we have local coordinates \((z_1,\dots, z_n)\) on some open subset \(U\) such that \(D_i=(z_i=0)\)  in \(U\) for \(1\leqslant i\leqslant c\). Then the above description shows that one has a local frame given by 
\[\left(\left(\frac{dz_1}{z_1}\right)^{\ell_1}\cdots \left(\frac{dz_c}{z_c}\right)^{\ell_c}dz_1^{i_1}\cdots dz_n^{i_n}\right)_{\substack{0\leqslant \ell_c \leqslant k\\ \ell_1+\cdots+\ell_c+i_1+\cdots+i_n=m}}.\]
Observe that \(S^m\Omega_X^1(\log D)_{(0),D_c}=S^m\Omega_X^1(\log D')\) and \(S^m\Omega_X^1(\log D)_{(m),D_c}=S^m\Omega_X^1(\log D)\).

Recall that one classically has the residue exact sequence
\[0\to \Omega_X^1(\log D')\to \Omega_X^1(\log D)\to \mathcal{O}_{D_c}\to 0,\]
which with our notation also reads
\[0\to S^1\Omega_X^1(\log D)_{(0),D_c}\to S^1\Omega_X^1(\log D)_{(1),D_c}\to \mathcal{O}_{D_c}\to 0.\]

This can be generalized as follows.

\begin{proposition}\label{prop:ExactSequencePole}
    Same notation as above. For any \(1\leqslant k\leqslant m\) one has an exact sequence
    \[0\to S^m\Omega_X^1(\log D)_{(k-1),D_c}\to S^m\Omega_X^1(\log D)_{(k),D_c}\to S^{m-k}\Omega_X^1(\log D')|_{D_c}\to 0.\]
\end{proposition}
\begin{proof}
    Consider an open subset \(U\) on which, for any \(1\leqslant i\leqslant c\), \(D_i\) is given by \((\sigma_i=0)\) for some \(\sigma_i\in \mathcal{O}(U)\). And define the map \[\rho_U:S^m\Omega_X^1(\log D)_{(k),D_c}(U)\to S^{m-k}\Omega_X^1(\log D')|_{D_c}(U)\]
    as follows: for any \(\omega\in S^m\Omega_X^1(\log D)_{(k),D_c}(U)\) one writes 
    \[\omega=\left(\frac{d\sigma_c}{\sigma_c}\right)^k\alpha+\beta,\]
    and then define
    \[\rho_U(\omega):=\alpha|_{D_c}.\]
    Let us see that this map is well defined. Suppose that one has two decompositions
    \[\left(\frac{d\sigma_c}{\sigma_c}\right)^k\alpha+\beta=\omega=\left(\frac{d\sigma_c}{\sigma_c}\right)^k\gamma+\delta\]
    where \(\alpha,\gamma\in S^{m-k}\Omega_X^1(\log D')(U)\) and \(\beta,\delta\in S^m\Omega_X^1(\log D)_{(k-1)}\). Then 
    \[\left(\frac{d\sigma_c}{\sigma_c}\right)^k(\alpha-\gamma)+(\beta-\delta)=0.\]
    But \(\sigma_c^{k-1}(\beta-\delta)\in  S^m\Omega_X^1(\log D')(U)\), and \(\sigma_c^k(\beta-\delta)\) vanishes along \(D_c\). Therefore
    \[(d\sigma_c)^k(\alpha-\gamma)=\sigma_c^{k}\left(\frac{d\sigma_c}{\sigma_c}\right)^k(\alpha-\gamma)=-\sigma_c^k(\beta-\delta)\]
also vanishes along \(D_c\), which means that \((\alpha-\gamma)|_{D_c}=0\) and thus
\[\alpha|_{D_c}=\gamma|_{D_c}.\]
We will now prove that the locally defined \(\rho_U\) glue to a sheaf morphism. Suppose we are given two open subsets \(U_1,U_2\), \(\sigma_{c,1}\in\mathcal{O}(U_1) \) and \(\sigma_{c,2}\in \mathcal{O}(U_2)\) such that \(D_c\cap U_1=(\sigma_{c,1}=0)\) and \(D_c\cap U_2=(\sigma_{c,2}=0)\). On \(U_{12}=U_1\cap U_2\) there exists \(g\in \mathcal{O}(U_{12})^*\) such that \(\sigma_{c,1}=g\sigma_{c,2}.\) Suppose now we are given \(\omega\in S^m\Omega_X^1(\log D)_{(k)}(U_{12})\) and we write
\[\omega=\left(\frac{d\sigma_{c,1}}{\sigma_{c,1}}\right)^k\alpha_1+\beta_1=\left(\frac{d\sigma_{c,2}}{\sigma_{c,2}}\right)^k\alpha_2+\beta_2.\]
Then using the formula
\[\frac{d\sigma_{c,1}}{\sigma_{c,1}}=\frac{d(g\sigma_{c,2})}{g\sigma_{c,2}}=\frac{d\sigma_{c,2}}{\sigma_{c,2}}+\frac{dg}{g},\]
one finds 
    \begin{eqnarray*}
        \left(\frac{d\sigma_{c,1}}{\sigma_{c,1}}\right)^k\alpha_1+\beta_1&=&
        \left(\frac{d\sigma_{c,2}}{\sigma_{c,2}}+\frac{dg}{g}\right)^k\alpha_1+\beta_1\\
        &=&\left(\frac{d\sigma_{c,2}}{\sigma_{c,2}}\right)^k\alpha_1+\sum_{\ell=0}^{k-1}\binom{k}{\ell}\left(\frac{d\sigma_{c,2}}{\sigma_{c,2}}\right)^\ell\left(\frac{dg}{g}\right)^{k-\ell}\alpha_1+\beta_1.
    \end{eqnarray*}
    Since
    \[\beta_1':=\sum_{\ell=0}^{k-1}\binom{k}{\ell}\left(\frac{d\sigma_{c,2}}{\sigma_{c,2}}\right)^\ell\left(\frac{dg}{g}\right)^{k-\ell}\alpha_1+\beta_1\in S^m\Omega_X^1(\log D)_{(k-1)}\] 
    the above argument shows that 
    \[\rho_{U_2}(\omega)=\alpha_2|_{D_c}=\alpha_1|_{D_c}=\rho_{U_1}(\omega)\]
    which proves that those locally defined map glue into a global map \(\rho\).

    It remains to show that \(\rho\)  is surjective and that \(\ker(\rho)=S^m\Omega_X^1(\log D)_{(k-1)}\) which are readily verified from the local description.
    
\end{proof}


Considering the long exact sequences in cohomology associated to the higher degree residue exact sequences, one obtains a way of proving vanishing results for logarithmic symmetric differential forms. 

\begin{proposition}\label{prop:CohomologyInequality}
    Same notation as above. Then for any \(1\leqslant k \leqslant m\), one has 
    \[h^0(X,S^{m}\Omega_X^1(\log D)_{(k),D_c})\leqslant h^0(X,S^m\Omega_X^1(\log D'))+\sum_{i=1}^kh^0(D_c,S^{m-i}\Omega_X^1(\log D')|_{D_c}).\]
\end{proposition}

Observe that  for \(i=m\), \(S^{m-i}\Omega_X^1(\log D')|_{D_c}\simeq \mathcal{O}_{D_c}\), hence \(h^0(D_c,S^{m-i}\Omega_X^1(\log D')|_{D_c})=1 \neq 0\). Hence we deduce the following. 
\begin{corollary}\label{cor:VanishingLogGeneral}
    Let \(X\) be a projective manifold and \(D=D_1+\cdots+D_c=D'+D_c\) be a simple normal crossing divisor on \(X\). For any \(m\geqslant 1\), if \(h^0(D_c,S^{m-i}\Omega_X^1(\log D')|_{D_c})=0\) for all \(1\leqslant i< m\), then
    \[h^0(X,S^m\Omega_X^1(\log D))\leqslant  h^0(X,S^{m}\Omega_X^1(\log D'))+1, \quad \text{and}\] 
    \[h^0(X,S^m\Omega_X^1(\log D)_{(k),D_c})=  h^0(X,S^{m}\Omega_X^1(\log D'))\quad \forall k<m.\]
In particular, if \(h^0(X,S^{m}\Omega_X^1(\log D'))=0\), then \(h^0(X,S^m\Omega_X^1(\log D))\leqslant 1\) and moreover, for all \(k<m\), one has  \(h^0(X,S^m\Omega_X^1(\log D)_{(k),D_c})=0\).
\end{corollary}
To obtain a proper vanishing, one can compute more precisely the coboundary map. Here we will do this  only in the one component case,  i.e. when  \(c=1\). In that case one has  \(D'=\varnothing\) and \(D_c=D\) and  for simplicity, we will use the notation \(S^m\Omega_X^1(\log D)_{(k)}:=S^m\Omega_X^1(\log D)_{(k),D}\).
\begin{lemma} \label{lemma}
     Suppose that \(D\) is a smooth irreducible divisor on \(X\). Let \(m>0\). Suppose that \(c_1(D)|_D\neq 0\). Consider the case \(k=m\) in the exact sequence given by Proposition \ref{prop:ExactSequencePole}.
    \[0\to S^m\Omega_X^1(\log D)_{(m-1)}\to S^m\Omega_X^1(\log D)\to \mathcal{O}_{D}\to 0.\]
    Then the induced coboundary map 
    \[\delta:H^0(D,\mathcal{O}_D)\to H^1(X,S^m\Omega_X^1(\log D)_{(m-1)})\]
    is non-zero, hence injective. In particular, we have
    \[H^0(X,S^m\Omega_X^1(\log D)_{(m-1)})= H^0(X,S^m\Omega_X^1(\log D)).\]
\end{lemma}
\begin{proof} For \(m=1\), the assertion follows directly from the usual residue sequence \[ 0\longrightarrow \Omega_X^1 \longrightarrow \Omega_X^1(\log D) \longrightarrow \mathcal O_D \longrightarrow 0, \] since the image of \(1\in H^0(D,\mathcal O_D)\) under the connecting morphism is the class \(c_1(\mathcal O_X(D))\), whose image in \(H^1(D,\Omega_X^1|_D)\) is nonzero by hypothesis. We may therefore assume \(m\geq 2\).

\

We will prove this by computing the coboundary map in Cech cohomology. Take a sufficiently refined cover \(\mathfrak{U}=(U_\alpha)_{\alpha\in \Lambda}\) such that on each chart \(U_\alpha\) the divisor \(D\) is given by an element \(\sigma_\alpha\in \mathcal{O}(U_\alpha).\) Let us denote the induced transition functions by \(g_{\alpha,\beta}\in \mathcal{O}(U_{\alpha,\beta})^*\) so that for every \(\alpha,\beta\in \Lambda\) one has
\[\sigma_\beta=g_{\alpha,\beta}\sigma_\alpha.\]
Since \(h^0(D,\mathcal{O}_D)=1\) it suffices to compute \(\delta(1)\). Over \(U_\alpha\) one can lift \(1\in \mathcal{O}_D(U_\alpha)\) to \[\left(\frac{d\sigma_\alpha}{\sigma_\alpha}\right)^m\in S^m\Omega_X^1(\log D)(U_\alpha).\] Therefore in Cech cohomology the cycle representing \(\delta(1)\) is given by \((\tau_{\alpha, \beta})_{\alpha,\beta\in \Lambda}\) where
\[\tau_{\alpha,\beta}=\left(\frac{d\sigma_\beta}{\sigma_\beta}\right)^m-\left(\frac{d\sigma_\alpha}{\sigma_\alpha}\right)^m.\]
But one has 
\[\frac{d\sigma_\beta}{\sigma_\beta}=\frac{d(g_{\alpha,\beta}\sigma_\alpha)}{g_{\alpha,\beta}\sigma_\alpha}=\frac{dg_{\alpha,\beta}}{g_{\alpha,\beta}}+\frac{d\sigma_\alpha}{\sigma_\alpha}.\]
Therefore 
\begin{eqnarray*}
\tau_{\alpha,\beta}&=&\left(\frac{d\sigma_\beta}{\sigma_\beta}\right)^m-\left(\frac{d\sigma_\alpha}{\sigma_\alpha}\right)^m=\left(\frac{dg_{\alpha,\beta}}{g_{\alpha,\beta}}+\frac{d\sigma_\alpha}{\sigma_\alpha}\right)^m-\left(\frac{d\sigma_\alpha}{\sigma_\alpha}\right)^m\\
&=&\left(\frac{d\sigma_\alpha}{\sigma_\alpha}\right)^m+\sum_{k=0}^{m-1}\binom{m}{k}\left(\frac{dg_{\alpha,\beta}}{g_{\alpha,\beta}}\right)^{m-k}\left(\frac{d\sigma_\alpha}{\sigma_\alpha}\right)^k-\left(\frac{d\sigma_\alpha}{\sigma_\alpha}\right)^m\\
&=& \sum_{k=0}^{m-1}\binom{m}{k}\left(\frac{dg_{\alpha,\beta}}{g_{\alpha,\beta}}\right)^{m-k}\left(\frac{d\sigma_\alpha}{\sigma_\alpha}\right)^k\\
&=& m\frac{dg_{\alpha,\beta}}{g_{\alpha,\beta}}\left(\frac{d\sigma_\alpha}{\sigma_\alpha}\right)^{m-1}  +\sum_{k=0}^{m-2}\binom{m}{k}\left(\frac{dg_{\alpha,\beta}}{g_{\alpha,\beta}}\right)^{m-k}\left(\frac{d\sigma_\alpha}{\sigma_\alpha}\right)^k.
\end{eqnarray*}
Now applying Proposition \ref{prop:ExactSequencePole} to \(k=m-1\) yields an exact sequence 
\[0\to S^m\Omega_X^1(\log D)_{(m-2)}\to S^m\Omega_X^1(\log D)_{(m-1)}\to \Omega_X^1|_D\to 0,\]
which induces in cohomology the map 
\[\rho:H^1(X,S^m\Omega_X^1(\log D)_{(m-1)})\to H^1(D,\Omega_X^1|_D).\]
But the explicit description given by Proposition \ref{prop:ExactSequencePole} and the above description of \(\delta(1)\) implies that in Cech cohomology, \(\rho(\delta(1))\) is given by the cocycle
\[\left(m\frac{dg_{\alpha,\beta}}{g_{\alpha,\beta}}\right)_{\alpha,\beta\in \Lambda},\]
which corresponds to the class
\[mc_1(D)\in H^1(D,\Omega_X^1|_D)\]
since the cocycle \(\left(\frac{dg_{\alpha,\beta}}{g_{\alpha,\beta}}\right)_{\alpha,\beta\in \Lambda}\) is a representative of \(c_1(D)\in H^1(X,\Omega_X^1).\)
But by hypothesis, the class \(\rho(\delta(1))=mc_1(D)\in H^1(D,\Omega_X^1|_D)\) is non-zero. In particular \(\delta(1)\) is non-zero, whence the result. The announced equality then follows at once by considering the long exact sequence associated in cohomology.

\end{proof}
From this one has the following result.
\begin{corollary}\label{cor:VanishingLog}
    Let \(X\) be a projective manifold and \(D\) be an irreducible smooth  divisor such that \(c_1(D)|_D\neq 0\) (e.g. if \(D\) is ample on \(X\)). For any \(m\geqslant 1\), if \(h^0(D,S^{m-k}\Omega_X^1|_D)=0\) for all \(1\leqslant k< m\), then
    \[h^0(X,S^m\Omega_X^1)= h^0(X,S^{m}\Omega_X^1(\log D)).\]  
    In particular, if \(X\) is moreover rationally connected, then \(h^0(X,S^{m}\Omega_X^1(\log D))=0\).
\end{corollary}
\begin{proof}
    This is a consequence of \cref{lemma} and \cref{prop:CohomologyInequality}. The last statement follows because \(h^0(X,S^m\Omega_X^1)=0\) for all \(m>0\) if \(X\) is rationally connected. 
\end{proof}
As an application, we prove some vanishing results that we shall need later. First, we show how we can easily derive a  special case of a classical result by Brückmann and Rackwitz.

\begin{corollary}\label{cor:VanishingP2}
    Let \(D\) be a smooth divisor in \(\PP^2\), then \(h^0(\PP^2,S^m\Omega_{\PP^2}^1(\log D))=0\) for all \(m>0\).
\end{corollary}
\begin{proof}
    By the above corollary, it suffices to prove that \(h^0(D,S^{m-k}\Omega_{\PP^2}^1|_D)=0\) for all \(1\leqslant k <m.\) By Euler's exact sequence \(\Omega_{\PP^2}^1\) is a subsheaf of \(\mathcal{O}_{\PP^2}(-1)^{\oplus 3}\) which is the dual of an ample vector bundle. In particular, for all \(0\leqslant k< m\),  \[h^0(D,S^{m-k}\Omega_{\PP^2}^1|_D)\leqslant h^0(D,S^{m-k}\mathcal{O}_{D}(-1)^{\oplus 3})=0.\]  
\end{proof}
More interestingly we prove a vanishing result on Hirzebruch surfaces. We recall some facts about Hirzebruch surfaces $\FF_d$, $d\geqslant 0$, which are the $\PP^1$-bundles over $\PP^1$, $\FF_d=\PP({\cal O}\oplus {\cal O}(d))$. Let $p:\FF_d \to \PP^1$ be the natural projection (the case $\FF_0=\PP^1\times \PP^1$ has two natural projections). The Neron-Severi group is

 $$NS(\FF_d)=\ZZ F\oplus \ZZ \Delta$$

 \noindent with $F$ the (numerical) class of the fibers of $p$ and $\Delta$ the class of a section of $p$ with $\Delta^2=d$. The intersection pairing is determined by 

 $$\Delta^2=d, \hspace {.5in} \Delta\cdot F=1, \hspace {.5in} F^2=0$$

 \noindent Set $\Delta_0=\Delta-dF$, which is the class of the unique section of $p$ with $\Delta_0^2=-d$, note that  $\Delta\cdot\Delta_0=0$. The following holds

 \begin{align}\label{can}
K_{\FF_d}=-2\Delta+(d-2)F \\ c_1(\Omega^1_{\FF_d/\PP^1})=-2\Delta+dF
\end{align}

 The first identity easily follows from adjunction for F and $\Delta$, while the second will then follow from:

\begin{align}\label{relseq}
0\to p^*\Omega^1_{\PP^1}\to \Omega^1_{\FF_d}\to \Omega^1_{\FF_d/\PP^1}\to 0
\end{align}

In Hartshorne V.2.18, there is a description of the possible classes in $NS(\FF_d)$ having irreducible, smooth and ample divisors of $\FF_d$,

\begin{align}\label{irred}\exists\, D \in \lvert a\Delta_0 + bF \rvert \ \text{irreducible or smooth}
\quad \Longleftrightarrow \quad
\begin{cases}
a = 1,\ b = 0,\\[4pt]
a = 0,\ b = 1,\\[4pt]
a > 0,\ b \geqslant ad \quad \text{if } d \geqslant 1,\\[4pt]
a > 0,\ b > 0 \quad \text{if } d = 0.
\end{cases}
\end{align}

The classes of ample divisors are the same as those of very ample divisors and require $a>0$ and $b>ad$.

\

\begin{theorem}\label{prop:VanishingFd1Comp} Let $D$ be a smooth and irreducible divisor of an Hirzebruch surface $\FF_d$. Then for $m\ge 1$

$$h^0(\FF_d, S^m\Omega^1_{\FF_d}(\log D))=0.$$

\end{theorem}

\begin{proof} There are three distinct cases to consider, D is not a fiber or $\Delta_0$, $D=F$ is a fiber and $D=\Delta_0$. 


\ 

\noindent Case: $D$ is not a fiber or $\Delta_0$.

\

By  \cref{cor:VanishingLog}, it suffices to prove that \(h^0(D,S^{m-k}\Omega_{\FF_d}^1|_D)=0\) for all \(1\leqslant k <m.\)

\

From the exact sequence \eqref{relseq}, and the fact that \(p^*\Omega_{\PP^1}^1\equiv -2F\) and \(\Omega_{\FF_d/\PP^1}^1\equiv -2\Delta +dF\) we see that 
    \[
    h^0(D,S^{m-k}\Omega_{\FF_d}^1|_D)\leqslant \sum_{\substack{p+q=m-k\\ p,q\geqslant 0}}h^0(D,p^*\Omega_{\PP^1}^{\otimes p}\otimes \Omega_{\FF_d/\PP^1}^{\otimes q})=\sum_{\substack{p+q=m-k\\ p,q\geqslant 0}}h^0(D,
    \mathcal{O}_{D}(-2pF+q(-2\Delta+dF)).
    \]
    Therefore, it suffices to prove the vanishing on the right hand side, which follows from the fact that \(\deg\mathcal{O}_{D}(-2pF+q(-2\Delta+dF))<0 \). Indeed, from the hypothesis, it follows that \(D\in |\mathcal{O}_{\FF_d}(a\Delta_0+bF)|\) with \(a>0\), \(b>0\) and \(b\geqslant ad\)
    \begin{eqnarray*}
        \deg\mathcal{O}_{D}(-2pF+q(-2\Delta+dF))&=&  D\cdot(-2pF+q(-2\Delta+dF))=(a\Delta_0+bF)\cdot ((dq-2p)F-2q\Delta)\\
        &=&a(dq-2p)-2bq =(ad-2b)q-2ap<0.
    \end{eqnarray*}
    In particular, from corollary \ref{cor:VanishingLog}, the announced vanishing follows if \(D\) is moreover ample. If \(D\) is irreducible and not ample, then \(b=ad\) (and also \(d\geqslant 1\)), hence \(D\cdot D= a^2(\Delta_0+dF)^2=a^2\Delta^2=a^2d\neq 0\). In particular \(c_1(D)|_D\neq 0\) and the desired vanishing follows again from corollary \ref{cor:VanishingLog}. 

    \

\noindent Case: $D=F$ is a fiber.

    \ 

    Since $F$ is the inverse image of the reduced divisor $p$, there is a logarithmic relative cotangent sequence

 $$0\to p^*\Omega^1_{\PP^1}(p)\to \Omega^1_{\FF_d}(\log F)\to \Omega^1_{\FF_d/\PP^1}\to 0$$

From the exact sequence follows a filtration for $S^m\Omega^1_{\FF_d}(\log F)$ whose graded pieces are (using $\Omega^1_{\PP^1}(p)\simeq \mathcal{O}_{\mathbb{P}^1}(-1)$):

\begin{align*} \operatorname{gr}_i &\simeq \bigl( \pi^*\mathcal{O}_{\mathbb{P}^1}(-1) \bigr)^{\otimes(m-i)} \otimes \bigl( \Omega_{\FF_d/\mathbb{P}^1} \bigr)^{\otimes i} \\ &\simeq \mathcal{O}_{\FF_d} \bigl(-2i\Delta_0-(id+m-i)f\bigr), \qquad 0\leq i\leq m. \end{align*}

For \(i=0\), we have \[ \operatorname{gr}_0 \simeq \pi^*\mathcal{O}_{\mathbb{P}^1}(-m), \] and hence \[ H^0(\FF_d,\operatorname{gr}_0)=0 \qquad (m>0). \]

\ 

For \(i>0\), restriction to a general fiber \(F\simeq\mathbb{P}^1\) gives \[ \left.\operatorname{gr}_i\right|_F \simeq \mathcal{O}_{\mathbb{P}^1}(-2i). \] Consequently, \[ H^0(\FF_d,\operatorname{gr}_i)=0. \] Indeed, a nonzero global section would restrict nontrivially to a general fiber, which is impossible because the restricted line bundle has negative degree. Since every graded piece has no nonzero global sections, induction along the filtration gives \[ H^0\!\left( \FF_d,S^m\Omega_{\FF_d}(\log F) \right)=0 \qquad\text{for every }m>0.\]

 \

 \noindent Case: $D=\Delta_0$

 \ 

We observe that we only need to consider case $d\ge 1$, since case $d=0$ is covered by $D=F$ ($\Delta_0=F$ for $d=0$). We have that \(\FF_d\setminus \Delta_0\) is covered by rational curves, and that moreover through any point \(x\in \FF_d\setminus \Delta_0\) and every tangent direction \(\xi\in T_{\FF_d,\xi}\) which is not tangent to the fiber of \(\pi:\FF_d\to \PP^1\) passing through \(x\), there exists a rational curve in \(\FF_d\setminus \Delta_0\) passing through \(x\) and tangent to \(\xi\). To see this, one must recall that there is a map \(\varphi : \FF_d\to \PP^n\) which contracts \(\Delta_0\) and whose image is the cone over the rational normal curve of degree \(d\), such that \(\varphi(\Delta_0)\) is the vertex of the cone.  Therefore, any hyperplane section not passing through the vertex will induce a rational curve on \(\FF_d\setminus \Delta_0\), from which the observation follows easily. 
    But since \(h^0(\PP^1,S^m\Omega_{\PP^1})=0,\) we obtain that for any rational curve \(f:\PP^1\to \FF_d\setminus \Delta_0\) and for any \(\omega\in H^0(\FF_d,S^m\Omega_{\FF_d}^1(\log \Delta_0))\), the restriction of \(f^*\omega=0\). Because the rational curves cover \(\FF_d\setminus\Delta_0\) and almost every tangent direction, we see that \(\omega=0\), thus \(h^0(\FF_d,S^m\Omega_{\FF_d}^1(\log(\Delta_0)))=0\).

\end{proof}
In the case of two components, one has the following  statement.

\

\begin{proposition}\label{prop:VanishingFd2Comp}
    Let \(D\subset \FF_d\) be a smooth irreducible divisor such that $D^2\neq 0$ and \(\Delta_0+D\) is a simple normal crossing divisor. Then, for any \(m>0\) one has
    \[h^0(\FF_d,S^m\Omega_{\FF_d}^1(\log(\Delta_0+D)))\leqslant 1\quad \text{and}\quad h^0(\FF_d,S^m\Omega_{\FF_d}^1(\log(\Delta_0+D))_{(m-1),D})=0.\]
\end{proposition}
\begin{proof}
    The assertion will follow from Corollary \ref{cor:VanishingLogGeneral} and \cref{prop:VanishingFd1Comp} (\(h^0(\FF_d,S^m\Omega_{\FF_d}^1(\log(\Delta_0)))=0\)) once we prove \(h^0(D,S^{m-i}\Omega_{\FF_d}^1(\log(\Delta_0))|_{D})=0\) for all \(1\le i<m\). 

\

     One has a commutative diagram

     $$ \xymatrix{
      & & 0 \ar[d] & 0 \ar[d]& \\
     0 \ar[r] & \pi^*\Omega^1_{\PP^1}\ar[r]\ar@{=}[d]& \Omega^1_{\FF_d}\ar[d]\ar[r]&\Omega^1_{\FF_d/\PP^1} \ar[d]\ar[r]& 0\\
       0 \ar[r] & \pi^*\Omega^1_{\PP^1}\ar[r]& 
     \Omega_{\FF_d}^1(\log \Delta_0)\ar[d]\ar[r]&\Omega^1_{\FF_d/\PP^1}(\log \Delta_0) \ar[d]\ar[r]& 0\\
     & &\mathcal{O}_{\Delta_0} \ar[d]\ar@{=}[r]&\mathcal{O}_{\Delta_0} \ar[d] & \\
     & & 0& 0&
  }$$
 Moreover, on the one hand we have \(\pi^*\Omega_{\PP^1}\simeq\pi^*\mathcal{O}_{\PP^1}(-2)\simeq\mathcal{O}_{\FF_d}(-2F)\) and \(\Omega_{\FF_d/\PP^1}(\log \Delta_0)=\Omega_{\FF_d/\PP^1}+\Delta_0\simeq\mathcal{O}_{\FF_d}(-2\Delta+dF+\Delta_0)\simeq \mathcal{O}_{\FF_d}(-\Delta)\). Thus we have an extension 
 \[0\to \mathcal{O}_{D}(-2F)\to \Omega_{\FF_d}^1(\log\Delta_0)|_{D}\to \mathcal{O}_{D}(-\Delta)\to 0.\]

 \
 
 In particular, we have for all  $ 1\le i < m$,
 \[h^0(D,S^{m-i}\Omega_{\FF_d}^1(\log\Delta_0)|_{D})\leqslant \sum_{\substack{p+q=m-i\\ p,q\geqslant 0}}h^0(D,
    \mathcal{O}_{D}(-2pF-q\Delta)).\]
   
By the assumptions \(D^2\neq 0\) and \(\Delta_0+D\) simple normal crossing, the cases \(D=F\), \(D=\Delta_0\), and, when \(d=0\), a section disjoint from \(\Delta_0\) in the same ruling are excluded. Hence with the addition of the description of irreducible classes in \(NS(\mathbb{F}_d)\) we have \(D\sim a\Delta_0+bF\) with \(a>0\) and \(b>0\) and  
    \[\deg_D\mathcal{O}_{D}(-2pF-q\Delta)=D\cdot(-2pF+-q\Delta)=(a\Delta_0+bF)\cdot(-2pF+-q\Delta)=-2ap-qb<0,\]

    From this we deduce \(h^0(D,
    \mathcal{O}_{D}(-2pF-q\Delta))\)=0, hence  \(h^0(D,S^{m-i}\Omega_{\FF_d}^1(\log\Delta_0)|_{D})=0\), as announced.
\end{proof}

\

\

\section{Minimal cotangent dimension for double covers}
\subsection{Symmetric differentials on double covers}
\ 

\

 We will now apply the results of the previous section to prove vanishing results for symmetric differential forms  for some special double covers.

  
\

\begin{lemma}\label{lemmavan} Let $Y$ be a smooth surface, $D=D_1+\cdots+D_c$ be a simple normal crossing divisor on \(Y\). Let \(\pi:X\to Y\) a double cover of \(Y\) with branch locus \(D\). Let \(\nu:\widehat{X}\to X\) be the minimal resolution of singularities of \(X\) and
$m>0$.  If $H^0(Y,S^{2m}\Omega^1_Y(\log D)_{(m),D_c})=0$, then 
$$H^0(\widehat{X},S^m\Omega^1_{\widehat{X}})=0$$
\end{lemma}

\begin{proof}  Let us denote by \(i:X\to X\) the involution on \(X\) as double cover, and \({\hat \imath}: \widehat{X}\to \widehat{X}\) the induced involution on \(\widehat{X} \). Let \(E\) denote the exceptional divisor of \(\nu\) and \(\Sigma\subset X\) the singularities of \(X\).  Pick $w\in H^0(\widehat{X},S^m\Omega^1_{\widehat{X}})$  and consider $w\otimes {\hat{\imath}}^*w\in H^0(\widehat{X},S^{2m}\Omega^1_{\widehat{X}})^{\ZZ_2}$ (the subspace of $\ZZ_2$-invariant $2m$-differentials). If $w$ is nontrivial, so is $w\otimes {\hat{\imath}}^*w$. 
We therefore obtain a \(\ZZ_2\) invariant section \[\omega^{\circ}:=w\otimes {\hat{\imath}}^*w|_{\widehat{X}\setminus E}\in H^0(\widehat{X}\setminus E,S^{2m}\Omega^1_{\widehat{X}\setminus E})^{\ZZ_2}\simeq H^0(X\setminus \Sigma,S^{2m}\Omega^1_{X\setminus \Sigma})^{\ZZ_2}.\]

Let us denote by \(\Sigma_Y\) the singularities of \(D\) in \(Y\). The invariant form \(\omega^{\circ}\) descends to a logarithmic form \(\mu^{\circ}\in H^0(Y\setminus \Sigma_Y,S^{2m}\Omega^1_{Y\setminus \Sigma_Y}(\log D))\). Moreover this form  has pole order at most \(m\) along any component of \(D\), and in particular along \(D_c\).
To see this, consider $R\subset X\setminus \Sigma$ the ramification locus of $\pi|_{X\setminus \Sigma}$ and any $x\in R$. There are local charts  $(V_{\pi(x)},u_1,u_2)$ and $(U_x,z_1,z_2)$, centered at $\pi(x)$ and $x$ respectively  such that $R\cap U_x=\{z_1=0\}$, \(D\cap V_{\pi(x)}=\{u_1=0\}\) and $\pi(z_1,z_2)=(z_1^2,z_2)$. 
In these coordinates, the involution \(i\) is given by \(i(z_1,z_2)=(-z_1,z_2)\) and \(\omega^\circ\) can be written as
\[\omega^\circ=\sum_{p+q=2m}a_{p,q}(z_1,z_2)(dz_1)^p(dz_2)^q.\]

Outside \(D\), using the relation \(dz_1=\frac{du_1}{2\sqrt{u_1}}\),
one has
\[\mu^\circ=\sum_{p+q=2m}a_{p,q}(\sqrt{u_1},u_2)\left(\frac{du_1}{2\sqrt{u_1}}\right)^p(du_2)^q.\]
On the other hand, since \(i^*\omega^\circ=\omega^\circ\), it follows that \(a_{p,q}(-z_1,z_2)=(-1)^pa_{p,q}(z_1,z_2)\). In particular, if \(p=2\ell\) is even, then only even powers of \(z_1\) appears in \(a(z_1,z_2)\), which implies that \(a_{p,q}(\sqrt{u_1},u_2)\) is holomorphic and one also has \(\left(\frac{du_1}{2\sqrt{u_1}}\right)^p(du_2)^q=\frac{(du_1)^p}{2u_1^\ell}(du_2)^q\) has logarithmic poles of order less than \(\ell=\frac{p}{2}\leqslant m\). If on the other hand \(p=2\ell+1\) is odd, then only odd powers of \(z_1\) appear in \(a_{p,q}(z_1,z_2)\) and in particular, \(\frac{a_{p,q}(\sqrt{u_1},u_2)}{\sqrt{u_1}}\) is holomorphic . On the other hand one has \(\left(\frac{du_1}{2\sqrt{u_1}}\right)^p(du_2)^q=\frac{(du_1)^p}{2u_1^\ell\sqrt{u_1}}(du_2)^q\), therefore 
\(a_{p,q}(\sqrt{u_1},u_2)\frac{(du_1)^p}{2u_1^\ell\sqrt{u_1}}(du_2)^q=\frac{a_{p,q}(\sqrt{u_1},u_2)}{\sqrt{u_1}}\frac{(du_1)^p}{2u_1^\ell}(du_2)^q\) has only logarithmic poles of order less than \(\ell\leqslant \frac{p-1}{2}\leqslant {m}\), as announced.

It now suffices to apply Hartogs theorem to extend  \(\mu^\circ\)  to a non-zero logarithmic form \(\mu\) on \(Y\) with only poles of order less than \(m\) along any component of \(D\) and in particular \(0\neq \mu\in H^0(Y,S^{2m}\Omega^1_Y(\log D)_{(m),D_c})=0\) contradicting our hypothesis.

\end{proof}




    

Combining this lemma with  \cref{cor:VanishingP2} and  \cref{prop:VanishingFd1Comp} we immediately deduce the following.
\begin{corollary}\label{cor:VanishingCover}
    Let \(X\) be a double cover of \(\PP^2\) or $\FF_d$ branched along a smooth irreducible curve,  then 
    \[H^0(X,S^m\Omega_X^1)=0,\ \ \forall m\geqslant 1.\]
\end{corollary}

 \

 \begin{corollary}\label{cor:VanishingCover2}
    Let \(X\) be the minimal resolution of a  double cover of \(\FF_d\) branched along a simple normal crossing divisor \(\Delta_0+D\) where \(D\) is an irreducible smooth curve. Then 
    \[H^0(X,S^m\Omega_X^1)=0,\ \ \forall m\geqslant 1.\]
\end{corollary}
\begin{proof} For $d>0$ the result follows from \cref{lemmavan} and  \cref{prop:VanishingFd2Comp} since the divisor $\Delta_0+F$ excluded in the proposition is not divisible by 2 in Pic($\FF_d)$ and hence cannot occur as the branch locus of a double cover. For $d=0$, the same approach yields the result, except for the case that $D$ is a distinct section of $\FF_0$, but in this case the double cover will still be a ruled surface over $\PP^1$ and hence the result holds.
    
\end{proof}

\

\begin{theorem}\label{A}
Let \(X\) be a smooth double cover of \(\PP^2\) or \(\FF_d\), then $X$ will not have maximal cotangent dimension (i.e. $\Omega^1_X$ cannot be big). Moreover, unless the base of the double cover is $\FF_d$ and the branch locus is a union of fibers, then $X$ has minimal cotangent dimension, i.e. 
    \[H^0(X,S^m\Omega_X^1)=0,\ \ \forall m\geqslant 1.\]

\end{theorem}

\begin{proof} If the base of the double cover is $\FF_d$ and the branch locus is a union of fibers, then the double cover $X$ is a ruled surface. In \cite{sakai}, it is shown that ruled surfaces have non-positive cotangent dimension. This follows from the easy addition formula of \cref{cd} and the fact that $\lambda(\Omega^1_X|_F,F)=\lambda({\cal O}_{\PP^1}\oplus {\cal O}_{\PP^1}(-2),\PP^1)=-1$, where $F$ is a fiber of the ruling on $X$.

\

Now assume that the branch locus $B$ is not a union of fibers. Since \(X\) is smooth, \(B\) has to be smooth as well, with possibly more than one irreducible component. If \(B\) has a single component, then the result follows directly from corollary \ref{cor:VanishingCover}. This covers the case where the base of the cover is $\PP^2$.

\ 

If the base of the cover is $\FF_d$, then the number of irreducible components of $B$ is at most two. In $\FF_d$ the irreducible curves with trivial self-intersection must be fibers (for $\FF_0$ fibers of one of the two rullings), with negative self-intersection there is only one, if it exists (i.e. d>0).  Fibers intersect  any curve that it is not a fiber. From this follows that if  $B$ has more than one component, then none of the components can be a fiber, since otherwise all the components had to be fibers and that is excluded by hypothesis. Finally, the observation follows from the Hodge index theorem, that excludes the existence of two components of $B$ with positive self-intersection that do not intersect. 

\

We need to consider the case with base of the double cover $\FF_d$ and branch locus $B$ with two components, from the above follows that \(B=\Delta_0+D\), where \(D\) is smooth irreducible and doesn't intersect \(\Delta_0\). We can then conclude from corollary \ref{cor:VanishingCover2}.

\end{proof}
\subsection{Maximal cotangent for double covers of rational surfaces}

\

We just saw that smooth double covers $X$ of minimal rational surfaces $Y$ have minimal cotangent dimension, with the possible exception of $Y=\FF_d$ with the branch locus $B$  a union of fibers. It is not difficult to see that if $Y=\FF_d$ and $B$ is a union of fibers the cotangent dimension cannot be maximal. This section illustrates that if we drop the minimality condition on the  rational surface that is the base of the double cover, then the cover can have maximal cotangent dimension. Below we provide a class of examples where the  cover has maximal cotangent dimension. Further examples will appear as Horikawa surfaces in the next section.

\begin{proposition}\label{CMS} There are smooth double covers of rational surfaces with maximal cotangent dimension.
    
\end{proposition}

\begin{proof} We apply the CMS-criterion \cref{cms} to show the existence of  minimal resolutions $X$ of double covers $X'$ of $\PP^2$, branched along simple curves (only a, d and e curve singularities) with many singularities, which have  maximal cotangent dimension.

\ 

Horikawa \cite{Horikawa1976I} showed that the minimal resolution $X$ of a double cover $X'$ of  $\PP^2$ branched along a simple curve is the same as the canonical resolution of $X'$, which is obtained by blowing up $\PP^2$ several times in order to resolve the singularities of the branch curve $B$ and then doing the double cover of the blown-up $\PP^2$ branched along the desingularization of $B$. Hence, $X$ can be viewed as a double cover of a rational surface. 

\ 

Let $\pi:X'_m\to \PP^2$ be the double cover of $\PP^2$ branched along a simple curve $B$ of degree $d=2m$ and $X_m$ be the minimal resolution of $X'_m$. The following holds for $X_m$,

\[ 
c_1^2(X_m)=2(m-3)^2
\]
\[
c_2(X_m)=2(2m^2-3m+3)
\]

Persson \cite{persson1982horikawa} showed that for $m\geqslant 4$, there are simple plane curves of degree $2m$ with $3m$ singularities of type $a_{m-1}$, this implies that $X'_m$ has $3m$ $A_{m-1}$ singularities. Using the lower bound $h^1_\Omega(A_n)>\frac {n}{6}-\frac{1}{45}$ given in \cref{h-bound},
it follows that:

\[
\sum_{x\in Sing(X'_{m})} h_\Omega^1(x)+\frac{s_2(X_m)}{3!}>\frac{1}{3!}(m^2-\frac{47}{5}m+12)
\]

\noindent note that $X'_m$ has ample canonical divisor and is the canonical model of $X_m$. The CMS-criterion gives that $X_m$ has maximal cotangent dimension for $m\geqslant 8$, when the right side is non-negative.

\end{proof}

\

\

\

\subsection{Generic Horikawa surfaces}

\

\

A (Noether) Horikawa surface for us is a minimal surface of general type lying on Noether's line, $c_1^2=2p_g-4$ (equivalently $c_1^2=2\chi -6$ or $c_1^2=\frac{1}{5}(c_2-36))$. The equivalences hold since all the conditions imply the irregularity $q=0$, see remark. 

\

\begin{remark}  Noether's inequality $c_1^2\geqslant 2p_g-4$ is equivalent to $c_1^2=2\chi+2q-6$ and $c_1^2\ge\frac{1}{5}(c_2-36+12q)$. Hence, if $c_1^2=2\chi-6$ or $c_1^2=\frac{1}{5}(c_2-36)$, we must have $q=0$. The result that  $c_1^2=2p_g-4$ implies $q=0$  appears in \cite {Bombieri}, theorem 9.  
\end{remark}

\

It follows from the above and the relation $12\chi=c_1^2+c_2$, that for Horikawa surfaces any one of the invariants $c_1^2$, $c_2$, $p_g$ and $\chi$ will determine the other three (and the Hilbert polynomial $P_K(m)=\chi(mK_X)$, which for a surface is determined $c_1^2$ and $\chi$). We will use $p_g$ as our reference invariant, $\chi=p_g+1$, $c_1^2=2p_g-4$ and $c_2=10p_g+16$. We denote by ${\cal H}_{p_g}$ the open subset parameterizing all smooth surfaces of  the quasi-projective Gieseker moduli scheme ${\cal M}_{2p_g-4,p_g+1}$. The quasi-projective variety ${\cal H}_{p_g}$ is irreducible if $ p_g \not\equiv 2 \!\pmod{4}$ and reducible with two irreducible and connected components otherwise.

\

Horikawa  studied the canonical map for Horikawa surfaces, $\phi_{|K_X|}:X\to \PP^{p_g-1}$ and stratified ${\cal H}_{p_g}$ in terms of the possible canonical images (or their minimal resolutions) which are $\PP^2$ or the Hirzebruch surfaces $\FF_d$. For
$p_g \neq 3,6$:

\[
{\cal H}_{p_g}
=
\bigsqcup_{\substack{d = 0 \\ d \equiv p_g \!\!\!\pmod{2}}}^{\left\lfloor \frac{p_g}{2} + 1 \right\rfloor}
{\cal H}_{p_g}^{(d)}.
\]

Here ${\cal H}_{p_g}^{(d)}$ parametrizes all Horikawa surfaces that are minimal
resolutions of double covers of $\mathbb{F}_d$ branched along a simple
curve (only singularities of type $a$, $d$, and $e$)
\[
B \in \left| 6\Delta + (p_g + 3d + 2)F \right|.
\]

For $ p_g \equiv 1 \!\pmod{2}$ (and $p_g\neq 3$) there is only one open stratum in ${\cal H}_{p_g}$ (the strata of higher dimension) which consists of ${\cal H}_{p_g}^{(1)}$; for $p_g \equiv 0 \!\pmod{2}$ with $ p_g \not\equiv 2 \!\pmod{4}$ there is also only one open stratum in ${\cal H}_{p_g}$ which consists of ${\cal H}_{p_g}^{(0)}$, while for $ p_g \equiv 2 \!\pmod{4}$,  there are two open strata in the reducible ${\cal H}_{p_g}$ which consist of ${\cal H}_{p_g}^{(0)}$ and ${\cal H}_{p_g}^{(\frac{p_2}{2}+1)}$, the latter stratum is one of the two irreducible components.

\

Concerning the special cases $p_g=3,6$;  ${\cal H}_3$ is irreducible and ${\cal H}_3={\cal H}_3^{(\infty)}$ which parametrizes all the minimal resolutions of double covers of $\mathbb{P}^2$ branched along a simple curve $B \in |8H|$;  ${\cal H}_6$ is reducible with two irreducible and connected components, one of the components is ${\cal H}_6^{(0)} \;\sqcup\; {\cal H}_6^{(2)}$ while the other is ${\cal H}_6^{(4)} \;\sqcup\; {\cal H}_6^{(\infty)}$,
where ${\cal H}_6^{(\infty)}$ parametrizes all the minimal resolutions of double covers
of $\mathbb{P}^2$ branched along a simple curve $B \in |10H|,$ the open strata are ${\cal H}_6^{(0)}$ and ${\cal H}_6^{(\infty)}$.

\

As a consequence of the precise description we have:

\begin{corollary}
    A generic Horikawa surface is either:
    \begin{enumerate}
        \item A double cover of \(\FF_d\) or \(\PP^2\) branched along a smooth ample divisor.
        \item The minimal resolution of a double cover of \(\FF_d\) branched along \(\Delta_0+D\), where $D$  is a smooth divisor intersecting \(\Delta_0\) transversely.
    \end{enumerate}
\end{corollary}
As a direct consequence of this description and of corollaries \ref{cor:VanishingCover} and \ref{cor:VanishingCover2}, we therefore obtain:

\begin{theorem} \label{generic} The generic Horikawa surface $X$ in any of the strata ${\cal H}_{p_g}^{(d)}$ of the Horikawa moduli spaces
${\cal H}_{p_g}$ satisfies
\[
\lambda(\Omega_X^1) = -\infty .
\]
\end{theorem}

\

\ 

\

\section{Maximal cotangent dimension and Campana fibrations of general type}

\ 
Recall the following statement appearing in \cite{sakai}.

\begin{proposition}[\cite{sakai}, Example $4$]\label{bignessfib}
Let $f: X \to C$ be a fibration on a projective smooth surface $X$ over a smooth curve $C$ with  $g(C) \geqslant 2$ and the general fiber of genus $\geqslant 2$. Then 
\[
\lambda(\Omega_X^1)=2.
\]
\end{proposition}

Sakai's fibration criterion for $\lambda(\Omega^1_X) = 2$ for a fibered surface of general type $f : X \to C$
requires the base curve $C$ to be of general type.  There are no such fibrations for minimal surfaces of general type in the region $K^2 \leqslant 2\chi$.  
As observed in \cref{intro}, if $K^2 \leqslant 2\chi$ the base $C$ of a fibration has $g(C) = 0,1$, moreover $g(C) = 0$ if $K^2 < 2\chi$.

\

Sakai's result can be extended to allow the base of the fibration not to be of general type, if one considers an orbifold structure on the base curve coming from the multiple fibers.  

\

Let $X$ be a fibered smooth surface over a smooth curve $C$, $f : X \to C$.
For a point $p \in C$, consider the fiber $F_p = f^{-1}(p)$ and its irreducible decomposition

\[
F_p = \sum_{i=1}^{k} n_i F_{p,i},
\]

where each $F_{p,i}$ is an irreducible component and $n_i$ is its multiplicity. 

\begin{definition} 
\begin{itemize}
\item The classical multiplicity of $F_p$ is
$m^*(F_p)=\gcd\{n_i\}$. A fiber is multiple if $m^*(F_p)\geqslant 2$.

\item The non-classical multiplicity of $F_p$ is $m(F_p)=\inf\{n_i\}$. If $m(F_p)\geqslant 2$, $F_p$ will be called a Campana multiple fiber.

\item The classical orbifold base of $f$ is the $\mathbb{Q}$-divisor
$$\Delta^*(f):=\sum_{c \in C} \left(1-\frac{1}{m^*(c)}\right)c.$$

\item The non-classical orbifold base of $f$ is the $\mathbb{Q}$-divisor
$$\Delta(f):=\sum_{c \in C} \left(1-\frac{1}{m(c)}\right)c.$$

\item The classical orbifold canonical divisor of the base $C$ of the fibered surface $X$ is
\[
K^*_{C,\mathrm{orb}}
=
K_C+\Delta^*(f).
\]
\item The non-classical orbifold canonical divisor of the base $C$ of the fibered surface $X$ is
\[
K_{C,\mathrm{orb}}
=
K_C+\Delta(f).
\]
\item The base $C$ is of classical orbifold general type if
\[
\kappa(C,K^*_{C,\mathrm{orb}})=1.
\]
\item The base $C$ is of non-classical orbifold general type if
\[
\kappa(C,K_{C,\mathrm{orb}})=1.
\] 
\item The fibration $f:X\to C$ is a Campana fibration of general type if $C$ is of non-classical orbifold general type

\end{itemize}
\end{definition}

\medskip

We start by considering the extension of Sakai's result  to fibrations whose base $C$ is of classical orbifold general type and observe that such an extension still can not hold for surfaces  with $K^2 < 2\chi$.

\begin{proposition} \label{orbsakai}
Let $f: X \to C$ be a fibration of a smooth projective surface $X$ over a base $C$ of classical orbifold general type which has the general fiber of genus $\geqslant 2$. Then
\[
\lambda(\Omega_X^1)=2.
\]
\end{proposition}

\begin{proof} Let $\{F_{p_i}\}_{i=1,\dots,\ell}$ be the multiple fibers of $f$, with
\[
m_i=m^*(F_{p_i}).
\]
Since the pair $(C, \Delta^*(f))$ is of general type, we can find a branched cover $g\colon C'\to C$ ramified only over the $p_i$ and with ramification indices over the $p_i$ equal to $m_i$ \cite{Nam87} (Thm 1.2.15). We say that $g\colon C'\to C$ is orbifold-\'etale.
Let $X'$ be the normalization of the fiber product $C' \times_C X$.
The surface $X'$ is smooth, \'etale over $X$ and fibered over $C'$. By construction
\[
K_{C'} = g^*\!\left( K_C + \sum_{i=1}^\ell \left(1-\frac{1}{m_i}\right)x_i \right)
       = g^*(K^*_{C,\mathrm{orb}}).
\]

Hence $C'$ is of general type and Sakai's fibration theorem gives that $X'$
has $\lambda(\Omega_{X'}^1)=2$. $\lambda(\Omega_X^1)=2$ follows from the
\'etale invariance of the cotangent dimension.
\end{proof}

\medskip

This result can be applied to surfaces of general type $X$  fibered over $C=\mathbb{P}^1$, but it cannot hold if $X$ is in the region $K^2 \leqslant 2\chi$.
If the base $C$ is of classical orbifold general type, we saw in the proof above
that $X$ has the irregular \'etale cover $X'$. The slope
$K^2/\chi$ is an \'etale invariant and the irregularity of $X'$ implies $X$ lies in the region
$\frac{K^2}{\chi}> 2.$

\

The extension below of the previous proposition to the non-classical case is much more interesting since it can be applied in cases where $K^2 < 2\chi$. As in the proof of Sakai's criterion, the key Lemma \cite{sakai} is the following.

\begin{lemma}\label{fib}
Let $f: X \to Y$ be a fibration between compact complex manifolds. Let $E$, $F$ be vector bundles on $X$ and $Y$ respectively. Assume that $\lambda(F,Y)= \dim Y$ and that there is a generically injective homomorphism $\sigma : f^*F \to E$. Then
$$\lambda(E,X)=\dim Y + \lambda(E_y, X_y)$$ 
\end{lemma}

We will use the following particular case. 
\begin{lemma}\label{fib2}
Let $f: X \to Y$ be a fibration between compact complex manifolds. Let $E$ be a vector bundle on $X$ and $F$ line bundle on $Y$. Assume that $\lambda(F,Y)= \dim Y$ and that there is a generically injective homomorphism $\sigma : f^*F \to S^m E$. Then
$$\lambda(E,X)=\dim Y + \lambda(E_y, X_y)$$ 
\end{lemma}

\begin{proof}
Consider the fibration $g: \PP(E) \to Y$. Then $\sigma$ gives a section of $\mathcal{O}_{\mathbb{P}(E)}(m) \otimes g^*F^{-1}$ or a generically injective morphism $g^*F \to \mathcal{O}_{\mathbb{P}(E)}(m)$. Applying Lemma \ref{fib} to the fibration $g$, we obtain $\lambda(E,X)=\dim Y + \lambda(E_y, X_y)$. 
    
\end{proof}

Now we prove the orbifold Sakai's criterion. 

\begin{proposition}\label{orbibig}
Let $f: X \to C$ be a fibration with non-classical orbifold general type base on a projective surface $X$ with general fiber of genus $\geqslant 2$. Then
\[
\lambda(\Omega_X^1)=2.
\]
\end{proposition}

\begin{proof}
Let $f: X \to C$ be a fibration over a base of non-classical orbifold general type on a projective surface $X$ with general fiber of genus $\geqslant 2$. From the definition of the pair $(C, \Delta(f))$, we see that we have a morphism from $f^*(\ell (K_C+\Delta(f)))$ to $S^\ell \Omega_X^1$, for any $\ell\in\NN$ such that $\ell(K_C+\Delta(f))$ is integral. 
Indeed, let \(p\in C\) be a point for which \(m_p>1\), and let \(t\) be a local coordinate on \(C\) centered at \(p\). A local generator of \(\mathcal O_C\bigl(\ell(K_C+\Delta(f))\bigr)\) near \(p\) is \[ \frac{(dt)^\ell}{t^{\ell(1-1/m_p)}}. \] 
Let \(x\) be a general point of an irreducible component \(F_{p,j}\) of the fiber over $p$ whose multiplicity is $n_{p,j}$. Since \(X\) is smooth, we may choose a local coordinate \(z\) on \(X\) centered at \(x\), transverse to \(F_{p,j}\), such that \(F_{p,j}=(z=0)\) and \[ f^*t=u z^{n_{p,j}}, \] where \(u\) is a holomorphic unit. Consequently, \[ f^*(dt) = z^{n_{p,j}-1} \bigl(n_{p,j}u\,dz+z\,du\bigr). \] It follows that \[ f^*\left( \frac{(dt)^\ell}{t^{\ell(1-1/m_p)}} \right) = z^{\ell(n_{p,j}/m_p-1)} \frac{\bigl(n_{p,j}u\,dz+z\,du\bigr)^\ell} {u^{\ell(1-1/m_p)}}. \] Since \(n_{p,j}\geq m_p\), we have $\ell\left(\frac{n_{p,j}}{m_p}-1\right)\geq 0$. Therefore, the above pullback is holomorphic at a general point of every irreducible component of every fiber of \(f\). Thus the resulting rational morphism $f^*\mathcal O_C\bigl(\ell(K_C+\Delta(f))\bigr) \dashrightarrow S^\ell\Omega_X^1$ is regular away from a subset of codimension at least two. Since \(X\) is smooth and both sheaves are locally free, it extends uniquely to a morphism on all of \(X\): \[ f^*\mathcal O_C\bigl(\ell(K_C+\Delta(f))\bigr) \longrightarrow S^\ell\Omega_X^1. \] This morphism is generically injective, since over the smooth locus of \(f\) it is induced by the injective morphism $f^*\Omega_C^1\longrightarrow\Omega_X^1.$

\ 

Since the pair $(C, \Delta(f))$ is of general type, we have $\lambda(\ell(K_C+\Delta(f)),C)=1$. From \cite{sakai} (Example $4$), we know that on a general fiber $X_c$, $\lambda(\Omega^1_{X|X_c}, X_c)=1$. We conclude by applying Lemma \ref{fib2}.
\end{proof}

\begin{remark}
The key property used in the proof is that $f: X \to (C, \Delta(f))$ is a $C$-morphism in the sense of \cite{KR1} i.e. a morphism which pulls-back ``orbifold'' differential forms. An alternative proof (close to the one given for Proposition \ref{orbsakai}) of Proposition \ref{orbibig} could be given using the properties of $C$-morphisms \cite[Proposition~14.9]{KR1}. 
\end{remark}

\
    
\

\

\subsection{Genus 2 fibrations and Horikawa surfaces with maximal cotangent dimension}

\ 

\ 

This section addresses the existence of minimal surfaces of general type in the region $K^2<2\chi$ having a Campana fibration of general type. We direct our results towards the most challenging case, Horikawa surfaces (lowest ratio $K^2/\chi$  for given $\chi$, $\chi\geqslant 4$). We approach this problem by considering surfaces with a genus 2 fibration. Stoppino in \cite{Stoppino} showed that there are Campana fibrations of general type of genus $2$, the examples produced were in the region $K^2\geqslant 2\chi$. We show that these fibrations exist in the region $K^2<2\chi$
and, in fact, can occur on special Horikawa surfaces.

\medskip

 To build genus $2$ fibrations with Campana multiple fibers we need first a
description of the possible Campana multiple genus $2$ fibers. Note that adjunction trivially implies that
Campana multiple fibers of genus 2 cannot be classical multiple fibers. 

\

There are three classifications of the possible genus 2 fibers. The initial one, by
Ogg \cite{ogg}, separates the genus 2 fibers in 44  numerical types (describes the irreducible components of the fiber, how they intersect and their multiplicities). Some of these numerical types contain more than a single possibility. In this description, there are four types, type (17), (18), (29) and (30), consisting of Campana multiple fibers of genus 2 and all have Campana multiplicity 2 (this was already observed by Stoppino \cite{Stoppino}). 

\ 

The second classification, by Namikawa--Ueno \cite{nu}, separates the genus 2 fibers in around 120 types which belong to five groups. This classification takes into account the geometry of the pencil near the fiber, in particular, it  describes how a genus $2$
fiber can appear in a local genus 2 fibration $f:{\cal X}\to \Delta$  that is the relatively minimal resolution of a double cover $g:{\cal X'}\to\Delta \times \mathbb{P}^1$ 
\[
\begin{tikzcd}
{\cal X}\arrow[r, "r"] \arrow[d, "f"'] & {\cal X'} \arrow[d, "g"] \\
\Delta & \Delta \times \mathbb{P}^1 \arrow[l]
\end{tikzcd}
\]

\noindent From this description, we obtain for each Campana multiple fiber an equation of the
branch locus $B\subset \Delta \times \mathbb{P}^1$ for a double cover $g:{\cal X'}\to\Delta \times \mathbb{P}^1$ that produces the Campana multiple fiber as the central fiber $f^{-1}(0)$.

\

Finally, there is the classification by Horikawa \cite{genus2} which separates the genus 2 fibers in six types, this description is geared towards understanding  the global properties of a projective surface $X$ with a genus 2 fibration having specific genus 2 fibers. More precisely, this description specifies the impact of a given genus 2 fiber on the geographic pair $(\chi,K^2)$ of the surface $X$, which is of interest to us. 

\

We expand on the Horikawa approach. In \cite{genus2}, it is shown that if $X$ is a minimal smooth regular surface with a genus $2$ fibration, $f\colon X \to \PP^1$, then one can have a diagram as follows:

\[\begin{tikzcd} 
& \widehat{X} \arrow[ld, "\phi"'] \arrow[rd, "r_c"] & \\
X \arrow[d, "f"'] & & X' \arrow[d, "g"] \\
\PP^1  & &\FF_d \arrow[ll, crossing over, "p"]
\end{tikzcd}
\]
\vspace {-1in}
\begin{align} \label{dia}
\end{align}

\vspace {.5in}

\noindent where: 

\begin{itemize}
    \item [i)] $p$ is the natural projection.
    \item [ii)] $g$ is a double cover of an Hirzebruch surface $\FF_d$,
branched along a reduced even sextic curve $B$, i.e. $B\in |6\Delta_0+kF|$ with $k$ even, with the collection of singularities of $B$ possible in a single fiber are restricted to only six types (see below).
    \item [iii)] $r_c$ is the canonical resolution of $X'$.
    \item[iv)] $\phi$ is a birational morphism contracting the $(-1)$-curves introduced in the fibers by the canonical resolution procedure (the number of (-1)-curves introduced determines the pair $(\chi,K^2)$ for $X$).
\end{itemize}

\

In \cite{genus2}, the classification of the singular genus 2 fibers comes from a classification of the singularities of the branch locus $B$ appearing on a single fiber of $\FF_d$. A fiber $\Gamma$ of $p$ is called singular if it contains singularities of $B$. Note that each singular fiber of the genus 2 fibration $f:X\to \PP^1$ corresponds to a singular fiber $\Gamma$ of $p$. Using the elementary transformations of Nagata on ruled surfaces, Horikawa showed that we can reduce the possible singular fibers $\Gamma$ of $p$ (possible configuration of singularities of $B$ on $\Gamma$) to six equivalence classes, hence six types. 

\ 

 The six types of Horikawa for singular genus 2 fibers are denoted by: $(0)$, $(I_k)$, $(II_k)$, $(III_k)$, $(IV_k)$ and $(V)$, with $k\geqslant 1$. We are only interested in the types $(0)$ and $(V)$, since the Campana multiple fibers of genus 2 belong to these types.  The type $(0)$ corresponds to fibers on which $B$ only has simple curve singularities, Ogg's numerical type (29) is of type (0). The type $(V)$ corresponds to  fibers $\Gamma$ that are contained in $B$ and such that $B_1=B-\Gamma$ has a quadruple point $x$ on $\Gamma$, which after  blowing up  $x$ the proper transform of $B_1$ has a double point on the proper transform of $\Gamma$. The numerical types (17), (18) and (30) are of type (V).

\  

Theorem 3 of \cite{genus2} states that for $X$ as above with the diagram \cref{dia}, the following relation between $K^2$ and $\chi$ holds:

\begin{align}\label{relation}
    K^2=2\chi-6+\nu(V)+\sum_k[(2k-1)(\nu(I_k)+\nu(III_k))+2k(\nu(II_k)+\nu(IV_k))]
\end{align}
\noindent where $\nu(*)$ is the number of singular fibers of $f$ of type $(*)$.

\

A direct corollary of this theorem is that  $X$ is a Horikawa surface ($K^2=2\chi-6$) if and only if all  the singular fibers of $f$ are of type $(0)$.

\

We give a  description of the Campana multiple fibers of genus 2, with a focus on the Campana multiple fiber of Horikawa type (0), denoted from now on simply by type (0). The Campana multiple fibers of type (0) are of special importance as they are the only ones possible on Horikawa surfaces by \cref{relation}.

\medskip
\medskip
\noindent
\textbf{Campana multiple fiber of type (0).} 

\ 

We consider the genus 2 fibers of numerical type (29) (\cite{ogg}), its curve configuration    is as follows

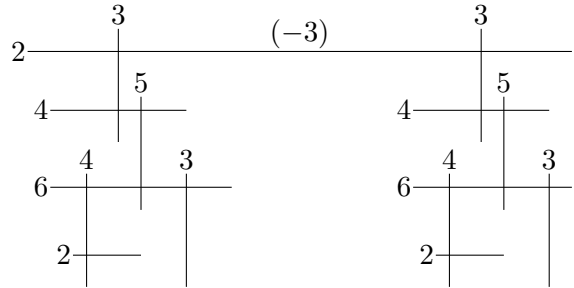
\[
\begin{tikzpicture}[scale=.6]

\draw (-6,3) -- (6,3);
\node at (-6.2,3) {$2$};

\node at (0,3.4) {$(-3)$};

\draw (-4,3.5) -- (-4,1);
\draw (4,3.5) -- (4,1);

\node at (-4,3.8) {$3$};
\node at (4,3.8) {$3$};

\draw (-5.5,1.7) -- (-2.5,1.7);
\draw (2.5,1.7) -- (5.5,1.7);

\node at (-5.7,1.7) {$4$};
\node at (2.3,1.7) {$4$};

\draw (-3.5,2) -- (-3.5,-0.5);
\draw (4.5,2) -- (4.5,-0.5);

\node at (-3.5,2.3) {$5$};
\node at (4.5,2.3) {$5$};

\draw (-5.5,0) -- (-1.5,0);
\draw (2.5,0) -- (6.0,0);

\node at (-5.7,0) {$6$};
\node at (2.3,0) {$6$};

\draw (-4.7,0.3) -- (-4.7,-2.2);
\node at (-4.7,0.6) {$4$};

\draw (-5,-1.5) -- (-3.5,-1.5);
\node at (-5.2,-1.5) {$2$};

\draw (-2.5,0.3) -- (-2.5,-2.2);
\node at (-2.5,0.6) {$3$};


\draw (3.3,0.3) -- (3.3,-2.2);
\node at (3.3,0.6) {$4$};

\draw (3,-1.5) -- (4.5,-1.5);
\node at (2.8,-1.5) {$2$};

\draw (5.5,0.3) -- (5.5,-2.2);
\node at (5.5,0.6) {$3$};
\end{tikzpicture}
\]

\noindent where all curves are smooth rational curves, all curves are (-2)-curves except for the one marked as a (-3)-curve, and the numbers attached are the corresponding multiplicities.

\ 

A possible branch locus $B$ of a double cover of $\Delta \times \mathbb{P}^1$, with coordinates $(t,[x:z])$,
providing a local genus $2$ fibration with a singular fiber of type (0) over
$t=0$ \cite{nu}, is

\[
B=\left\{\, t\,(x^3-t^2z^3)\bigl((x-z)^3-t^2z^3\bigr)=0 \,\right\}.
\]

The branch locus has three irreducible components $B=\Gamma+B_1+B_2$, where $\Gamma$ is the central fiber and $B_i$, $i=1,2$,  have a cusp on $\Gamma$  with  the order of contact of the $B_i$ with the central fiber $\Gamma.B_i=3$.

\[
    \begin{tikzpicture}[scale=.6][font=\footnotesize]

\draw (0,-3) -- (0,3);
\node [font=\small] at (0,3.3) { $\Gamma$};
\draw[dashed] (-.6,-3) -- (-.6,3);
\draw[dashed] (-1,-3) -- (-1,3);

\draw[dashed] (.6,-3) -- (.6,3);
\draw[dashed] (1,-3) -- (1,3);


\draw (0,.5) .. controls (-0.4,2) and (-0.8,2.2) .. (-1.6,2.4);
\draw (0,.5) .. controls (-0.4,2.1) and (-0.8,2.3) .. (-1.6,2.5);
\draw (0,.5) .. controls (-0.4,1.9) and (-0.8,2.1) .. (-1.6,2.3);
\draw (0,.5) .. controls (0.4,2) and (0.8,2.2) .. (1.6,2.4);
\draw (0,.5) .. controls (0.4,2.1) and (0.8,2.3) .. (1.6,2.5);
\draw (0,.5) .. controls (0.4,1.9) and (0.8,2.1) .. (1.6,2.3);
\node at (-1.9,2.4) {$B_1$};

\draw (0,-1.5) .. controls (-0.4,-0.1) and (-0.8,0.1) .. (-1.6,0.3);
\draw (0,-1.5) .. controls (-0.4,0) and (-0.8,0.2) .. (-1.6,0.4);
\draw (0,-1.5) .. controls (-0.4,0.1) and (-0.8,0.3) .. (-1.6,0.5);
\draw (0,-1.5) .. controls (0.4,-0.1) and (0.8,0.1) .. (1.6,0.3);
\draw (0,-1.5) .. controls (0.4,0) and (0.8,0.2) .. (1.6,0.4);
\draw (0,-1.5) .. controls (0.4,0.1) and (0.8,0.3) .. (1.6,0.5);
\node at (-1.9,0.4) {$B_2$};
\end{tikzpicture}
\]
\hspace {3in} {\bf Figure 1.}

\

The branch locus $B$ has two singularities, both are the (simple curve singularity) $e_7$, hence it is of Horikawa type $(0)$. 

\medskip
\medskip
\noindent
\textbf{Campana multiple fibers of type (V).} 

\ 

\ 

We describe the remaining three possible numerical types of Campana multiple genus 2 fibers, types (17), (18) and (30), via the equation of a branch locus $B\subset \Delta \times \mathbb{P}^1$ for which the associated double cover has the Campana multiple fiber as the fiber over $0\in \Delta$, these equations appear in \cite{nu}. The curve configuration of these fibers are described in both \cite{ogg} and \cite{nu}.

\

\

a) The  numerical type (17) corresponds to the genus 2 fiber that comes, as described previously, from the following branch locus:

\[
B=\left\{\, t\,(x^6+tx^3z^3+t^2z^6)=0 \,\right\}\subset \Delta\times \PP^1
\]

The branch locus has three irreducible components $B=\Gamma+B_1+B_2$, where $\Gamma$ is the central fiber, $B_1=\left\{\, x^3+\frac{1+\sqrt{3}i}{2}tz^3=0 \,\right\}$ and $B_2=\left\{\, x^3+\frac{1-\sqrt{3}i}{2}tz^3=0 \,\right\}$. The irreducible components of $B$ are smooth meeting at $x=(0,[0,1])$, making $x$ the only singularity of $B$  a triple point of $B$. The triple point $x$ is not simple since after one blow up at $x$, the strict transforms of the irreducible components still meet at a triple point. 

\ 

 b)  The numerical type (18) contains infinitely many distinct genus 2 curves, one for each pair $(k,\ell)$ with $k\geqslant 0$ and $\ell=0,1,2$ (we remark, that we kept the separation of type (17) from type (18) as in \cite{ogg} and \cite{nu}, but type (17) is the case $k,l=0$ of type (18)). The equation of the branch locus for a given pair $(k,\ell)$ is

\[
B=\left\{\, t\,(x^3-tz^3)2+t^{k+2}x^\ell z^{6-\ell})=0 \,\right\}\subset \Delta\times \PP^1
\]

The number of irreducible components of $B$ is either two or three, but the relevant feature is that it follows from the equation that  $\Gamma$  is contained in $B$ and $Supp(B-\Gamma)\cap\Gamma=\{x\}$ with $x$ a non-simple triple point of $B$. 

\  

c) The  numerical type (30) corresponds to the genus 2 fiber that comes from the following branch locus:

\[
B=\left\{\, t\,(x^2+tz^2)(x^4+tz^4)=0 \,\right\}\subset \Delta\times \PP^1
\]

It is clear from the equation for $B$ that $B$ has three smooth irreducible components, $B=\Gamma+B_1+B_2$, meeting at $x=(0,[0,1])$, making $x$ the only singularity of $B$  a triple point of $B$. The triple point $x$ is not simple since after one blow up at $x$, the strict transforms of the irreducible components still meet at a triple point.

\

We just saw that all the Campana multiple fibers of types (17), (18) and (30) come from a fiber $\Gamma$ over $t=0$ that is contained in $B$ and $Supp(B-\Gamma)\cap\Gamma=\{x\}$ with $x$ a non-simple triple point of $B$. In the proof of Lemma 6 of \cite{genus2}, Horikawa   shows these fibers are   of type (V) using a single elementary transformation.

\

\

\begin{theorem} \label{max} There are Horikawa surfaces with maximal cotangent dimension. In particular, we show that this is possible for $p_g=5k-2$ for any $k\geqslant 5$.
    
\end{theorem}
\begin{proof} We will show that there are Horikawa surfaces $X$ that are genus-2 fibrations with at least five Campana fibers of type (0) and hence,  by \cref{orbibig}, with $\lambda_\Omega(X)=2$.

\

 The result holds if we can construct in some $\FF_d$ an even sextic curve $B$ with: 1) only simple singularities, hence by the paragraph of  \cref{relation}, the minimal resolution $X$ of the double cover associated with $B$ is a Horikawa surface and 2)   having the configuration given in figure 1, associated to the Campana multiple fiber of type (0), in at least five fibers of $\FF_d$ guaranteeing  $\Omega^1_X$ big by \cref{orbibig}.

\

 The configuration of the branch locus $B$ near the fiber $\Gamma$ of $\FF_d$ associated with the Campana multiple fiber of type (0) consists of three irreducible components, the fiber $\Gamma$ and two curves  having a cusp at different points of $\Gamma$ such that the multiplicity of intersection of the fiber and each of the two curves is three. To produce such $B$ start with two cuspidal cubic curves, $C_1$ and $C_2$, in $\PP^2$  with cusps at $x_i \in C_i$ and such that the line $\ell_0$ passing through the $x_1$ and $x_2$ satisfies for $i=1,2$
\[
m_{x_i}(C_i,\ell_0)=3.
\]

The example we consider for the construction has
\[
C_1=\{z_0^3+z_1^2 z_2=0\}
\quad\text{and}\quad
C_2=\{(z_0-z_2)^3+z_1^2 z_2=0\}.
\]

\noindent with cusps $x_1=[0:0:1]$ and $x_2=[1:0:1]$. The intersection \(C_1\cap C_2\) consists of five points at which both curves are smooth.  At four of these points, $p_i$, $i=1,...,4$, $m_{p_i}(C_1,C_2)=1$, while at the remaining point $p_5=[0:1:0]$  we have $m_{p_5}(C_1,C_2)=5$. This implies that the singularities of $C=C_1+C_2$ are the two cusps at $\{x_i\}_{i=1,2}$, four $a_1$ singularities at $\{p_i\}_{i=1,...,4}$ and one $a_9$ singularity at $p_5$.

\medskip

Let
\[
B_0=C+\ell_0+\ell_1
\]
with $\ell_0=\{z_1=0\}$ and $\ell_1$ a generic line. The singularities of $B_0$ are
\[
\mathrm{Sing}(B_0)=\{x_1,x_2,p_1,\dots,p_5,q,q_1,\dots,q_6\},
\]
where
\[
\{q\}=\ell_0\cap \ell_1
\quad\text{and}\quad
\{q_1,\dots,q_6\}=\ell_1\cap C.
\]

The curve $B_0\subset \PP^2$ is simple, as its singularities are of type ADE. $B_0$ has two $e_7$ singularities,
i.e. with local equation $x^3y+y^3=0$, at $x_1$ and $x_2$, with the line
$\ell_0$ as the smooth branch of both singularities. Less relevant to us,
$B_0$ has eleven $a_1$ singularities at $p_1,\dots,p_4,q,q_1,\dots,q_6$
and one $a_9$ singularity at $p_5$.

\medskip

Consider the blow up of $\mathbb{P}^2$ at $q$,

\[
\begin{tikzcd}
\FF_1  \arrow[d, "p"']\simeq \mathrm{Bl}_q \mathbb{P}^2 \arrow[r, "\sigma"] & \mathbb{P}^2 \\
\mathbb{P}^1 &
\end{tikzcd}
\]

\noindent and let $B$ be the strict transform of $B_0$,
\[
B=\sigma^*B_0-2\Delta_0=\Gamma_0+\Gamma_1+C_1'+C_2' \in |6\Delta_0+8F|,
\]
where $\Gamma_i=p^{-1}(t_i)$ are fibers of $p$ corresponding to
the strict transforms of $\ell_i$ and $C_i'$ are the transforms of the $C_i$.

\medskip

The curve $B \subset \FF_1$ is also a simple curve with the same singularities as $B_0$ except for the absence of one of the $a_1$, that is resolved when we blow up at $q$. The two $e_7$ singularities lie in the fiber $\Gamma_0$
and by construction $B$ has at $\Gamma_0$ the configuration appearing in figure 1 and for which the double cover associated to $B$ gives a Campana genus 2 fiber of type (0) over $\Gamma_0$.

\ 

Let $X_1'$ be the double cover of $\FF_1$ branched along $B$
(there is no ambiguity about the line bundle $L$, $L^{\otimes 2}\simeq \mathcal{O}(B)$)
and $X_1$ be the minimal resolution of $X_1'$, with diagram
\[
\begin{tikzcd}
X_1 \arrow[r, "r"] \arrow[dr, "f"'] & X_1' \arrow[r, "\,g\,"] & \FF_1 \arrow[dl, "p"] \\
 & \mathbb{P}^1 &
\end{tikzcd}
\]

\noindent the map $g$ is the double cover and $r$ the minimal resolution. By construction, $f:X_1 \to \PP^1$ has a single Campana multiple fiber of type $(0)$, the fiber $f^{-1}(t_0)$.

\medskip

We parametrize $\mathbb{P}^1$ such that the fibers
$f^{-1}([1\!:\!0])$ and $f^{-1}([0\!:\!1])$ are smooth and let
$\pi_k\colon \mathbb{P}^1 \to \mathbb{P}^1$ be the $k$-cyclic cover of $\mathbb{P}^1$ branched over
$[1\!:\!0]$ and $[0\!:\!1]$ (\,$\pi_k([x_0\!:\!x_1])=[x_0^k:x_1^k]$\,). Consider the base change of $p\colon \FF_1\to\mathbb{P}^1$
associated to $\pi_k$, as in \cite{persson1982horikawa}:
\[
\begin{tikzcd}
\FF_k \arrow[r, "\phi_k"] \arrow[d, "p_k"'] & \FF_1 \arrow[d, "p"] \\
\mathbb{P}^1 \arrow[r, "\pi_k"] & \mathbb{P}^1
\end{tikzcd}
\]

\noindent   where $p_k$ is the natural projection. The curve
\[
B_k=\phi_k^*B
\]
is simple and belongs to
\[
|6\Delta'_0+8kF|
\]
\noindent with $\Delta'_0$ the unique section in $\FF_k$ with negative self-intersection. Since the singularities of $B$ lie in the \'etale locus of $\phi_k$,
there are $k$ isomorphic singularities over each singularity of $B$. More relevant, we get the $k$ fibers of $p_k$ over the points in $\pi_k^{-1}(t_0)$ where $B_k$ gives the local
configuration of the branch locus associated to the Campana multiple fiber  of type~(0).
\

If we do the double cover of $\FF_k$ branched along $B_k$, we have the diagram:

\begin{center}
\begin{tikzpicture}[
    node distance=1cm and 2cm, 
    auto, 
    >=Stealth, 
    black, 
    every node/.style={inner sep=2pt}
]
    \node (X1) {$X_k$};
    \node (Xk) [right=of X1] {$X_1$};
    
    \node (X1p) [below=of X1] {$X_k'$};
    \node (Xkp) [below=of Xk] {$X_1'$};
    
    \node (F1) [below=of X1p] {$\mathbb{F}_k$};
    \node (Fk) [below=of Xkp] {$\mathbb{F}_1$};
    
    \node (P1l) [below=of F1] {$\mathbb{P}^1$};
    \node (P1r) [below=of Fk] {$\mathbb{P}^1$};
    
    \draw[<-] (Xk) -- node[above] {$g_k$} (X1);
    \draw[<-] (Xkp) -- (X1p);
    \draw[<-] (Fk) -- node[above] {$\phi_k$} (F1);
    \draw[<-] (P1r) -- node[above] {$\pi_k$} (P1l);
    
    \draw[->] (X1) -- node[left] {$r_k$} (X1p);
    \draw[->] (X1p) -- node[left] {$g_k$} (F1);
    \draw[->] (F1) -- node[left] {$p_k$} (P1l);
    
    \draw[->] (Xk) -- node[right] {$r$} (Xkp);
    \draw[->] (Xkp) -- node[right] {$g$} (Fk);
    \draw[->] (Fk) -- node[right] {$p$} (P1r);
    
    \draw[->] (X1.west) to [bend right=35] node[left] {$f_k$} (P1l.west);
    \draw[->] (Xk.east) to [bend left=35] node[right] {$f$} (P1r.east);

\end{tikzpicture}

\end{center}

\noindent where the map $g_k$ is the double cover of $\FF_k$ with branch locus $B_k$, $r_k$ the minimal resolution of $X_k'$. 
\medskip

The   fibrations $f_k:X_k \to \PP^1$ are genus 2 fibrations with $k$ Campana multiple singular fibers of type~(0). If $k\geqslant 5$, the base of $f_k$  is of non-classical orbifold general type and hence, by \cref{orbibig}, the surfaces $X_k$ have maximal cotangent dimension. 

\ 

As observed before, \cref{relation} gives that the $X_k$ are Horikawa surfaces since all singular fibers are of type (0) (or equivalently the $X_k$ are minimal resolutions of double covers of  Hirzebruch surfaces branched along a sextic curve with only simple singularities). Below we will  calculate the invariants of $X_k$.

\ 

First, we observe that by the simultaneous resolution results of Brieskorn applied to families of double covers with branch locus acquiring simple singularities (see \cite{persson1982horikawa} section 1 or  \cite{Horikawa1976I} theorem 3.1), it follows that there are small deformations $X'_k$ of $X_k$ that are double covers of $\FF_k$ with smooth branch locus in the same linear system. Hence, the invariants $c_1^2$ and $\chi$ of $X_k$ are the same as the ones for $X'_k$, 
the double cover of $\FF_k$, $\pi \colon X \to \FF_k$, branched along a smooth 
\[
B \in |6\Delta_0 + 8kF|
\]

\medskip

The formulas for the invariants of double covers, see for example \cite{BPV} V.22, give:
\[
c_1^2(X_k) = 2\left(K_{\FF_k} + \frac{B}{2}\right)^2,
\]
\[
\chi(\mathcal{O}_{X_k})
= 2\chi(\mathcal{O}_{\FF_k})
+ \frac{1}{2}\left(K_{\FF_k}\cdot \frac{B}{2} + \left(\frac{B}{2}\right)^2\right).
\]

Now use
\[
K_{\FF_k}^2 = 8,\hspace{.2in} \chi(\mathcal{O}_{F_k})=1,\hspace{.2in}
K_{\FF_k}\cdot B = -10k - 12\hspace{.1in} \text{and} \hspace{.1in} B^2 = 60k,
\]
and obtain:
\[
c_1^2(X_k) = 10k - 8,
\]
\[
\chi(\mathcal{O}_{X_k}) = 5k - 1.
\]

\noindent as expected respecting the relation $c_1^2(X_k) = 2\chi(\mathcal{O}_{X_k}) - 6$. In particular, it follows that the Horikawa surfaces obtained by this construction have $p_g=5k-2$ for any $k\geqslant 5$.
\end{proof}

\begin{remark} The surface $X_k$ is of general type and minimal since it is the minimal resolution  of $X_k'$ which has only canonical singularities, hence the minimal resolution $r_k:X_k \to X'_k$ is crepant, $r_k^*K_{X'_k}=K_{X_k}$,  and
\[
K_{X'_k} = \pi^*\!\left(K_{\FF_k} + \frac{B}{2}\right)
= \pi^*\!\bigl(\Delta_0 + (3k-2)F\bigr),
\]
is ample
for $k \geqslant 2$, see \cref{irred}, and for $k=1$ $K_{X'_1}$ is nef with  $K^2_{X_1}=2$.
    
\end{remark}

\ 

\begin{remark} As in the CMS-criterion,  the maximal cotangent dimension of the Horikawa surfaces $X_k$ just constructed is due to the singularities present  in their canonical models, at least they need 10 $E_7$ singularities to guarantee the five fibers of type (0). In contrast to the CMS-criterion, the number of $E_7$ singularities is not enough for the orbifold generalization of Sakai's criterion, \cref{orbibig}, to hold, the global geometric arrangement of the $E_7$ singularities plays an essential part.
    
\end{remark}

\ 

\begin{remark} \cref{fails} states that a minimal surface $X$ with $K^2/\chi\leqslant 2$ whose canonical model only has singularities of type $A$ cannot satisfy the CMS-criterion (there is an expectation that this statement is true in general). The examples $X_k$ just constructed have singularities of type $E$ in their canonical models. We observe that for example the Horikawa surface $X_5$ does not satisfy the CMS-criterion. The singularities of the canonical model of $X_5$ are 10 $E_7$, 50 $A_1$ and 5 $A_9$. Although we do not know the invariant $h^1_\Omega(E_7)$, we know that it satisfies the bound $h^1_\Omega(E_7)<\frac{c_2(E_7)}{3!}=\frac{383}{288}$, where $c_2(E_7)$ is the local second Chern class of the singularity $E_7$, see \cite{ADOWI} and \cite{ADOWII}. For the $A_n$ singularities we know the formula of $h^1_\Omega(A_n)$, \cite{ADOWII}, but it is simpler to use the bound $h^1_\Omega(A_n)<\frac{n}{6}$ derived from the formula. Hence, we have the bound:

$$Lh^1_\Omega(X_5)=\sum_{x\in Sing(X_{5,can})}h^1_\Omega(x)<10\frac{383}{288}+50\frac{1}{6}+5\frac{9}{6}=29+\frac{19}{144}$$

\noindent while 

$$\frac{c_1^2(X_5)-c_2(X_5)}{3!}=-34$$
    
\end{remark}

\

\

\begin{proposition} Let $X \in \mathcal{H}_{p_g}$ be a genus $2$-fibration with non-classical orbifold base of general type. Then $p_g \geqslant 8$.
    
\end{proposition}

\begin{proof} There are two relevant bounds for the number \(\rho_{E_7}(X)\) of disjoint \(E_7\)-configurations of \((-2)\)-curves on a minimal surface of general type \(X\) with fixed Chern numbers \((c_1^2,c_2)\).

\medskip

\noindent
(\emph{Miyaoka bound})
\[
\rho_{E_7}(X)\left(8-\frac{1}{48}\right) \leqslant c_2 - \frac{1}{3}c_1^2.
\]

\medskip

\noindent
(\emph{Hodge bound})
\[
\rho_{E_7}(X) \leqslant \frac{1}{7}\left(\frac{5c_2 - c_1^2}{6} + b_1 - 1\right).
\]

\medskip

\noindent

We observe that while Miyaoka bounds the number of $E_7$-configurations of (-2)-curves, the Hodge bound does not take into account the configurations formed by the (-2)-curves, it just bounds the number of $(-2)$-curves, \cite{miyaoka1984maximal} and \cite{ADOWI}.

\medskip

For Horikawa surfaces the Hodge bound is stronger and gives
\[
\rho_{E_7}(X) \leqslant \frac{8}{7}p_g + \frac{83}{42},
\]
The result follows since for a Horikawa surface to have a fibration with a non-classical orbifold base of general type it needs at least 5 fibers of Campana multiple fibers of type $(0)$ and hence at least ten $E_7$-configurations of (-2)-curves. 
    
\end{proof}

\begin{remark} Although it is known that Horikawa surfaces are simply connected, we observe that the examples considered here of genus $2$-fibration with non-classical orbifold base of general type are easily seen to be simply connected. Indeed, these fibrations have no classical multiple fibers and the Campana multiple fibers of type~(0) present are simply connected: their irreducible components are smooth rational curves and their dual graph is a tree. Since the base is \(\mathbb P^1\), it follows from Lemma 5.8 of \cite{Campana05} that \(X\) is simply connected.
\end{remark}

\

By Bogomolov-McQuillan's results, genus $2$-fibration with non-classical orbifold base of general type satisfy the Green-Griffiths-Lang conjecture. Nevertheless, in this setting we can precisely describe the exceptional locus containing entire curves.

\begin{proposition}\label{locus}
Let $f: X \to \PP^1$ be a fibration with non-classical orbifold general type base on a projective surface $X$ with general fiber of genus $\geqslant 2$. Then all entire curves are contained in the singular fibers.
\end{proposition}
\begin{proof}
Consider the fibration $f: X \to \PP^1$, its orbifold base $\Delta_f:=\sum_i (1-\frac{1}{m_i})p_i$ and an entire curve $g: \CC \to X$. From the definition of $\Delta_f$, the composed map $h:=f \circ g: \CC \to \PP^1$ is an entire curve with the property that if $h(t)=p_i$ then $\ord_t(h)\geqslant m_i$. From the generalization of little Picard theorem to maps with ramification \cite{Nev70}, one obtains that if the orbifold base is of general type, then $h$ is constant. Smooth fibers are of genus $\ge 2$ and hence hyperbolic, so the image of $g$ has to lie in a singular fiber.    
\end{proof}

\bibliographystyle{amsalpha}
\bibliography{references.bib}
\end{document}